\documentclass[11pt]{article}%
\usepackage[fleqn]{amsmath}
\usepackage{amsmath,amsfonts,amssymb,graphicx,makeidx,amsthm}
\usepackage{amscd,amsfonts,amssymb,multicol}
\usepackage{amssymb, amsmath}
\usepackage{amsfonts}
\usepackage{amssymb}
\usepackage{graphicx}%
\providecommand{\U}[1]{\protect\rule{.1in}{.1in}}
\newtheorem{theorem}{Theorem}[section]

\newtheorem{example}[theorem]{Example}

\newtheorem{lemma}[theorem]{Lemma}

\newtheorem{remark}[theorem]{Remark}

\numberwithin{equation}{section}
\newenvironment{Proof}[1][Proof]{\noindent\textbf{#1.} }{\ \rule{0.5em}{0.5em}}
\begin{document}

\title{Realization of the inverse scattering transform method for the Korteweg--de
Vries equation}
\author{Sergei M. Grudsky$^{1}$, Vladislav V. Kravchenko$^{2}$, Sergii M. Torba$^{2}$\\{\small $^{1}$ Departamento de Matem\'{a}ticas, Cinvestav, Cd. de M\'{e}xico,
07000 MEXICO}, \\{\small $^{2}$ Departamento de Matem\'{a}ticas, Cinvestav, Unidad
Quer\'{e}taro, }\\{\small Libramiento Norponiente \#2000, Fracc. Real de Juriquilla,
Quer\'{e}taro, Qro., 76230 MEXICO.}\\{\small e-mail: grudsky@math.cinvestav.mx, vkravchenko@math.cinvestav.edu.mx,
storba@math.cinvestav.edu.mx}}
\maketitle

\begin{abstract}
A method for practical realization of the inverse scattering transform method
for the Korteweg--de Vries equation is proposed. It is based on analytical
representations for Jost solutions and for integral kernels of transformation
operators obtained recently (\cite{Kr2019MMAS InverseAxis}, \cite{KrBook2020}%
). The representations have the form of functional series in which the first
coefficient plays a crucial role both in solving the direct scattering and the
inverse scattering problems. The direct scattering problem reduces to
computation of a number of the coefficients following a simple recurrent
integration procedure with a posterior calculation of scattering data \ by
well known formulas. The inverse scattering problem reduces to a system of
linear algebraic equations from which the first component of the solution
vector leads to the recovery of the potential. We prove the applicability of
the finite section method to the system of linear algebraic equations and
discuss numerical aspects of the proposed method. Numerical examples are
given, which reveal the accuracy and speed of the method.

\end{abstract}

\section{Introduction}

The inverse scattering transform method (ISTM) for solving nonlinear partial
differential equations was discovered in 1967 in the paper \cite{Gardner et al
1967} in application to the Korteweg--de Vries (KdV) equation, which can be
written in the form%
\[
u_{t}-6uu_{x}+u_{xxx}=0.
\]
The equation models shallow water waves and admits solitary wave solutions.

The method was further developed in 1972 in \cite{Zakharov et al 1972} in
application to another important equation of mathematical physics, the
nonlinear Schr\"{o}dinger equation, whose principal applications are to the
propagation of light in nonlinear optical fibers (see, e.g., \cite{Shaw}).

After those first works the ISTM encountered many other applications related
to a large number of nonlinear evolution partial differential equations of
mathematical physics. We refer to the books \cite{Ablowitz 2011},
\cite{Ablowitz Clarkson}, \cite{Ablowitz Segur} covering parts of this wide topic.

In theory the ISTM offers a beautiful way for solving the initial value
problem for the nonlinear equation. In order to obtain the solution at a given
time one has to solve successively a couple of spectral problems for a linear
ordinary differential equation (or system): a direct and an inverse scattering
problems. If one is interested in visualizing the evolution of the solution in
time, then not just one but a number of inverse scattering problems has to be solved.

The most important practical result of the ISTM until now is the possibility
to obtain solitonic solutions (see, e.g., \cite{Ablowitz 2011}, \cite{Ablowitz
Clarkson}, \cite{Ablowitz Segur}, \cite{Chadan}, \cite{Lamb}, \cite{Marchenko}%
). However, the ISTM is not widely used for numerical solution of the initial
value problems with general initial profiles (not necessarily leading to
solitonic solutions). The reason is the difficulty of solving the direct and
inverse scattering problems on the line. Instead of the ISTM many purely
numerical approaches have been developed for solving nonlinear evolutionary
equations, in particular, the KdV equation (see, e.g., \cite{Aksan et al},
\cite{Baker et al}, \cite{Courtes et al}, \cite{Dehghan et al}, \cite{Rashid}%
). They have important limitations. (a) Instead of considering problems on the
whole line with respect to the space variable, they solve problems on finite
intervals with some fictitious boundary conditions, and thus, in general, do
not serve for solving the initial value problem on the whole line. (b) In
order to obtain a picture of the solution at a given time the solution needs
to be computed for preceding times on a sufficiently fine mesh.

Recently, with the aid of several new ideas developed and implemented, there
appeared a feasible way for practical realization of the ISTM in all its might
and with all attractive features. This is the main subject of the present work.

It is well known that in principle solution of direct and inverse scattering
problems on the line for the one-dimensional Schr\"{o}dinger equation
\[
-y^{\prime\prime}+q(x)y=\rho^{2}y,\quad-\infty<x<\infty
\]
with $(1+\left\vert x\right\vert )q(x)\in L_{1}(-\infty,\infty)$, and $\rho
\in\mathbb{C}$, reduces to construction of the kernel of an appropriate
transmutation (transformation) operator, which maps the exponential function
into the Jost solution. When solving the direct problem, the knowledge of the
transmutation kernel implies the knowledge of the Jost solution for all values
of the spectral parameter, which in turn leads to the possibility of computing
scattering data. And solution of the inverse problem reduces to the
Gelfand-Levitan-Marchenko integral equation serving for calculation of the
transmutation kernel from given scattering data. The potential is recovered
from the transmutation kernel by one differentiation.

However, until recently the approach based on the transmutation kernel for
solving both direct and inverse scattering problems in general has not
resulted in an efficient numerical method. Construction of the transmutation
kernel by a known potential is quite a challenge as well as the recovery of
the transmutation kernel from the Gelfand-Levitan-Marchenko equation.

In the recent work \cite{Kr2019MMAS InverseAxis}, \cite{KrBook2020}, \cite{DKK
Jost} it was shown that solution of both direct and inverse scattering
problems can be reduced to calculation of the first coefficient of a
Fourier-Laguerre series expansion of the transmutation kernel. The kernel
itself is not needed. To the difference from the transmutation kernel which is
a function of two independent variables, the first coefficient is a function
of one variable. Its knowledge allows one to compute the Jost solution for all
values of the spectral parameter (which serves for solving the direct
problem), as well as to recover the potential (when solving the inverse
problem). More precisely, it is convenient to consider simultaneously two Jost
solutions, which we denote by $e(\rho,x)$ and $g(\rho,x)$, satisfying
corresponding asymptotic conditions at the \textquotedblleft$+$%
\textquotedblright\ and \textquotedblleft$-$\textquotedblright\ infinities,
respectively, as well as the first coefficients of Fourier-Laguerre series
expansions of two corresponding transmutation kernels. The first coefficients
we denote by $a_{0}(x)$ and $b_{0}(x)$, respectively. They can be written in
terms of the Jost solutions as follows%
\begin{equation}
a_{0}(x)=e(\frac{i}{2},x)e^{\frac{x}{2}}-1 \label{a0 Intro}%
\end{equation}
and
\begin{equation}
b_{0}(x)=g(\frac{i}{2},x)e^{-\frac{x}{2}}-1. \label{b0 Intro}%
\end{equation}
Thus, to know $a_{0}(x)$ is equivalent to know the Jost solution $e(\rho,x)$
for $\rho=\frac{i}{2}$ and similarly, to know $b_{0}(x)$ is equivalent to know
the Jost solution $g(\rho,x)$ for $\rho=\frac{i}{2}$. It is obvious that the
knowledge of $a_{0}(x)$ or $b_{0}(x)$ is sufficient for recovering $q$.
Indeed, from (\ref{a0 Intro}) and (\ref{b0 Intro}) it follows that%
\begin{equation}
q=\frac{a_{0}^{\prime\prime}-a_{0}^{\prime}}{a_{0}+1} \label{q=a0}%
\end{equation}
and
\begin{equation}
q=\frac{b_{0}^{\prime\prime}+b_{0}^{\prime}}{b_{0}+1}. \label{q=b0}%
\end{equation}
Moreover, in \cite{Kr2019MMAS InverseAxis}, \cite{KrBook2020} a system of
linear algebraic equations was derived for the coefficients of the
Fourier-Laguerre series of the transmutation kernel and thus $a_{0}(x)$ and
$b_{0}(x)$ can be obtained by solving the corresponding system of linear
algebraic equations.

Furthermore, as it was shown in \cite{DKK Jost} (see also \cite[Chapter
10]{KrBook2020}) the Jost solution $e(\rho,x)$ for all values of $\rho$,
$\operatorname{Im}\rho\geq0$ can be quite easily obtained from $a_{0}(x)$, and
similarly $g(\rho,x)$ from $b_{0}(x)$. Calculation of the Jost solutions from
$a_{0}(x)$ and $b_{0}(x)$ reduces to a recurrent integration procedure
\cite{DKK Jost}, \cite[Chapter 10]{KrBook2020}, which eventually leads to a
very convenient procedure for computing both the discrete scattering data and
the reflection coefficients by evaluating power series in a unitary disk of
the complex variable $z:=\left(  \frac{1}{2}+i\rho\right)  /\left(  \frac
{1}{2}-i\rho\right)  $. The overall approach based on the computation of the
functions $a_{0}(x)$ and $b_{0}(x)$ for solving both the direct and inverse
scattering problems leads to a direct, quite simple and efficient numerical
method for solving the Cauchy problem for the KdV equation. In the present
work we discuss this method in detail.

We provide a rigorous justification of the method which includes the
applicability of the finite section method, the analysis of the convergence
rate in dependence on the smoothness of the potential, the existence of the
derivatives $a_{0}^{\prime}$, $a_{0}^{\prime\prime}$ and $b_{0}^{\prime}$,
$b_{0}^{\prime\prime}$ (when $a_{0}$ and $b_{0}$ are computed as solutions of
the linear algebraic systems), which is required for recovering the potential
$q$ from (\ref{q=a0}) or (\ref{q=b0}), respectively, and discuss details of
its numerical implementation. Numerical examples are given, which reveal the
accuracy and speed of the method proposed.

\section{Representations for Jost solutions and their derivatives}

We consider the classical one-dimensional scattering problem. Given a real
valued function $q(x)$, $-\infty<x<\infty$ satisfying the condition
\begin{equation}
\int_{-\infty}^{\infty}\left(  1+\left\vert x\right\vert \right)  \left\vert
q(x)\right\vert dx<\infty, \label{cond q standard}%
\end{equation}
compute the corresponding scattering data which include a finite set of
eigenvalues and norming constants and a reflection coefficient. All the
scattering data are defined in terms of so-called Jost solutions of the
Schr\"{o}dinger equation%
\begin{equation}
-y^{\prime\prime}+q(x)y=\lambda y,\quad-\infty<x<\infty\label{Schr}%
\end{equation}
where $\lambda\in\mathbb{C}$ is a spectral parameter, $\lambda=\rho^{2}$,
$\rho\in\overline{\Omega}_{+}:=\left\{  \rho\in\mathbb{C}\mid\operatorname{Im}%
\rho\geq0\right\}  $. They are the unique solutions $e(\rho,x)$ and
$g(\rho,x)$ of (\ref{Schr}) satisfying the asymptotic relations
\begin{equation}
e(\rho,x)=e^{i\rho x}\left(  1+o(1)\right)  ,\quad e^{\prime}(\rho,x)=i\rho
e^{i\rho x}\left(  1+o(1)\right)  ,\quad x\rightarrow\infty, \label{asympt e}%
\end{equation}%
\[
g(\rho,x)=e^{-i\rho x}\left(  1+o(1)\right)  ,\quad g^{\prime}(\rho,x)=-i\rho
e^{-i\rho x}\left(  1+o(1)\right)  ,\quad x\rightarrow-\infty
\]
uniformly in $\overline{\Omega}_{+}$. When $\rho\in\mathbb{R}$ we have
\begin{equation}
e(-\rho,x)=\overline{e(\rho,x)},\quad g(-\rho,x)=\overline{g(\rho,x)}.
\label{Jost conjugation}%
\end{equation}

Under the condition (\ref{cond q standard}) the Jost solutions admit the
following integral representations%
\begin{equation}
e(\rho,x)=e^{i\rho x}+\int_{x}^{\infty}A(x,t)e^{i\rho t}dt \label{e}%
\end{equation}
and
\begin{equation}
g(\rho,x)=e^{-i\rho x}+\int_{-\infty}^{x}B(x,t)e^{-i\rho t}dt \label{g}%
\end{equation}
where $A$ and $B$ are real valued functions such that
\begin{equation}
A(x,x)=\frac{1}{2}\int_{x}^{\infty}q(t)dt, \label{A(x,x)}%
\end{equation}%
\[
B(x,x)=\frac{1}{2}\int_{-\infty}^{x}q(t)dt,
\]
$A(x,\cdot)\in L_{2}\left(  x,\infty\right)  $ and $B(x,\cdot)\in L_{2}\left(
-\infty,x\right)  $. More on the properties of the kernels $A$ and $B$ can be
found in \cite{Levitan}.

\begin{theorem}
\cite{Kr2019MMAS InverseAxis}, \cite{KrBook2020} The functions $A$ and $B$
admit the following series representations%
\begin{equation}
A(x,t)=\sum_{n=0}^{\infty}a_{n}(x)L_{n}(t-x)e^{\frac{x-t}{2}} \label{A series}%
\end{equation}
and
\begin{equation}
B(x,t)=\sum_{n=0}^{\infty}b_{n}(x)L_{n}(x-t)e^{-\frac{x-t}{2}}
\label{B series}%
\end{equation}
where $L_{n}$ stands for the Laguerre polynomial of order $n$.

For any $x\in\mathbb{R}$ fixed, the series converge in the norm of
$L_{2}\left(  x,\infty\right)  $ and $L_{2}\left(  -\infty,x\right)  $, respectively.

For the coefficients $a_{0}(x)$ and $b_{0}(x)$ the equalities are valid%
\begin{equation}
a_{0}(x)=e(\frac{i}{2},x)e^{\frac{x}{2}}-1 \label{a0}%
\end{equation}
and
\begin{equation}
b_{0}(x)=g(\frac{i}{2},x)e^{-\frac{x}{2}}-1. \label{b0}%
\end{equation}

The coefficients $a_{n}(x)$, $b_{n}(x)$, $n=0,1,\ldots$ are the unique
solutions of the equations%
\begin{equation}
La_{0}-a_{0}^{\prime}=q, \label{a0 eq}%
\end{equation}%
\begin{equation}
Lb_{0}+b_{0}^{\prime}=q, \label{b0 eq}%
\end{equation}%
\[
La_{n}-a_{n}^{\prime}=La_{n-1}+a_{n-1}^{\prime},\quad n=1,2,\ldots,
\]%
\[
Lb_{n}+b_{n}^{\prime}=Lb_{n-1}-b_{n-1}^{\prime},\quad n=1,2,\ldots,
\]
with $L:=\frac{d^{2}}{dx^{2}}-q(x)$, satisfying the boundary conditions
$a_{n}(x)=o(1)$, when $x\rightarrow+\infty$ and $b_{n}(x)=o(1)$, when
$x\rightarrow-\infty$.
\end{theorem}

In \cite{DKK Jost} (see also \cite{KrBook2020}) a recurrent integration
procedure for efficient computation of the coefficients of the series
(\ref{A series}) and (\ref{B series}) was derived. We give it in Appendix.

Note that due to the identity $L_{n}(0)=1$, for all $n=0,1,\ldots$, from
(\ref{A series}) and (\ref{B series}) two useful relations follow%
\[
\sum_{n=0}^{\infty}a_{n}(x)=A(x,x)=\frac{1}{2}\int_{x}^{\infty}q(t)dt,\quad
\sum_{n=0}^{\infty}b_{n}(x)=B(x,x)=\frac{1}{2}\int_{-\infty}^{x}q(t)dt.
\]

Denote
\begin{equation}
z:=\frac{\frac{1}{2}+i\rho}{\frac{1}{2}-i\rho}. \label{z}%
\end{equation}
Notice that this is a M\"{o}bius transformation of the upper halfplane of the
complex variable $\rho$ onto the unit disc $D=\left\{  z\in\mathbb{C}%
:\,\left\vert z\right\vert \leq1\right\}  $. In terms of the parameter $z$ the
following representations for the Jost solutions were obtained in
\cite{Kr2019MMAS InverseAxis} (see also \cite{DelgadoKhmelnytskayaKrHalfline}
and \cite{KrBook2020})%
\begin{equation}
e(\rho,x)=e^{i\rho x}\left(  1+\left(  z+1\right)  \sum_{n=0}^{\infty}\left(
-1\right)  ^{n}z^{n}a_{n}(x)\right)  \label{e via z}%
\end{equation}
and
\begin{equation}
g(\rho,x)=e^{-i\rho x}\left(  1+\left(  z+1\right)  \sum_{n=0}^{\infty}\left(
-1\right)  ^{n}z^{n}b_{n}(x)\right)  \label{g via z}%
\end{equation}
where the coefficients $a_{n}$ and $b_{n}$ can be constructed following the
recurrent integration procedure described in Appendix.

The coefficients $a_{n}$ and $b_{n}$ are real valued functions, $a_{n}%
(x)\rightarrow0$ when $x\rightarrow+\infty$ while $b_{n}(x)\rightarrow0$ when
$x\rightarrow-\infty$. For any $x\in\mathbb{R}$ the series $\sum_{n=0}%
^{\infty}a_{n}^{2}(x)$ and $\sum_{n=0}^{\infty}b_{n}^{2}(x)$ converge, which
is a consequence of the fact that they are Fourier coefficients with respect
to the system of Laguerre polynomials of corresponding functions from
$L_{2}\left(  0,\infty;e^{-t}\right)  $ (see
\cite{DelgadoKhmelnytskayaKrHalfline}). Hence for any $x\in\mathbb{R}$ the
functions $e(\rho,x)e^{-i\rho x}-1$ and $g(\rho,x)e^{i\rho x}-1$ belong to the
Hardy space $H^{2}(D)$ as functions of $z$ (this is due to the well known
result from complex analysis, see, e.g., \cite[Theorem 17.12]{Rudin}).

Analogous representations were obtained for the derivatives with respect to
$x$ of the Jost solutions under an additional assumption of the absolute
continuity of the potential $q$ (see \cite{DelgadoKhmelnytskayaKrHalfline}),
\begin{equation}
e^{\prime}(\rho,x)=e^{i\rho x}\left(  \frac{z-1}{2(z+1)}-\frac{1}{2}\int
_{x}^{\infty}q(t)dt+(z+1)\sum_{n=0}^{\infty}\left(  -1\right)  ^{n}z^{n}%
d_{n}(x)\right)  \label{e prime d}%
\end{equation}

\begin{equation}
g^{\prime}(\rho,x)=e^{-i\rho x}\left(  -\frac{z-1}{2(z+1)}+\frac{1}{2}%
\int_{-\infty}^{x}q(t)dt+(z+1)\sum_{n=0}^{\infty}\left(  -1\right)  ^{n}%
z^{n}c_{n}(x)\right)  \label{g prime c}%
\end{equation}
where the coefficients $\left\{  d_{n}\right\}  $ and $\left\{  c_{n}\right\}
$ are obtained from the coefficients $\left\{  a_{n}\right\}  $ and $\left\{
b_{n}\right\}  $, respectively, with the aid of the relations given in
Appendix. Here we emphasize that the whole procedure of computation of the
four sets of coefficients requires the computation of the Jost solutions
$e(\frac{i}{2},x)$, $g(\frac{i}{2},x)$ and their first derivatives. All
subsequent operations besides arithmetic operations involve the integration
only, which from a practical viewpoint makes the procedure convenient for
efficient computation. For the representations (\ref{e via z}%
)-(\ref{g prime c}) truncation error estimates were obtained in
\cite{DelgadoKhmelnytskayaKrHalfline}.

\begin{remark}
From (\ref{a0}) and (\ref{b0}) it follows that the knowledge of $a_{0}$ or
$b_{0}$ is sufficient for recovering $q$. Namely,
\begin{equation}
q=\frac{a_{0}^{\prime\prime}-a_{0}^{\prime}}{a_{0}+1} \label{q=a}%
\end{equation}
and
\begin{equation}
q=\frac{b_{0}^{\prime\prime}+b_{0}^{\prime}}{b_{0}+1}. \label{q=b}%
\end{equation}
Indeed, from (\ref{a0}) we have that $e(\frac{i}{2},x)=e^{-\frac{x}{2}}\left(
a_{0}(x)+1\right)  $. Differentiating twice this equality and recalling that
$e(\frac{i}{2},x)$ is a solution of the equation $-y^{\prime\prime
}+q(x)y=-\frac{1}{4}y$ gives us (\ref{q=a}). The equality (\ref{q=b}) is
obtained analogously.
\end{remark}

\section{Scattering data \label{section-Background}}

We refer to the books \cite{Chadan, Levitan, Marchenko, Yurko2007} for the
theory of the scattering problem. Here we introduce only the definitions
indispensable for the present work.

Consider the following scattering amplitudes (elements of the scattering
matrix) $a(\rho)=-\frac{1}{2i\rho}W\left[  e(\rho,x),g(\rho,x)\right]  $ and
$b(\rho)=\frac{1}{2i\rho}W\left[  e(\rho,x),g(-\rho,x)\right]  $, $\rho
\in\mathbb{R}$ where $W$ denotes the Wronskian. Notice that since the
Wronskian of any pair of solutions of (\ref{Schr}) is constant, we may
consider it at $x=0$. Thus,
\[
a(\rho)=-\frac{1}{2i\rho}W\left[  e(\rho,0),g(\rho,0)\right]  \quad\rho
\in\overline{\Omega_{+}},
\]
and
\[
b(\rho)=\frac{1}{2i\rho}W\left[  e(\rho,0),g(-\rho,0)\right]  ,\quad\rho
\in\mathbb{R}.
\]
Notice that the second expression is well defined for real values of $\rho$
only, because when $\operatorname{Im}\rho>0$, a solution of (\ref{Schr})
behaving like $e^{-i\rho x}$ when $x\rightarrow-\infty$ is not unique.

The reflection coefficients (right (+) and left (-)) have the form \cite[p.
210]{Yurko2007}
\begin{equation}
s^{\pm}(\rho):=\mp\frac{b(\mp\rho)}{a(\rho)},\quad\rho\in\mathbb{R}.
\label{reflection}%
\end{equation}

We recall that the eigenvalues of (\ref{Schr}), if they exist, form a finite
set of negative numbers $\lambda_{k}=\rho_{k}^{2}=\left(  i\tau_{k}\right)
^{2}$, $0<\tau_{1}<\ldots<\tau_{N}$. The corresponding norming constants are
introduced as follows%
\[
\alpha_{k}^{+}:=\left(  \int_{-\infty}^{\infty}e^{2}(\rho_{k},x)dx\right)
^{-1}\quad\text{and\quad}\alpha_{k}^{-}:=\left(  \int_{-\infty}^{\infty}%
g^{2}(\rho_{k},x)dx\right)  ^{-1}.
\]
For their computation another formula can be used \cite[p. 215]{Yurko2007}
\begin{equation}
\alpha_{k}^{\pm}=\frac{\left(  d_{k}\right)  ^{\pm1}}{ia^{\prime}(\rho_{k})}
\label{akpha_n}%
\end{equation}
where the constants $d_{k}$ are defined from the relation $d_{k}=g(\rho
_{k},x)/e(\rho_{k},x)=g(\rho_{k},0)/e(\rho_{k},0)$ (notice that when
$\lambda_{k}=\rho_{k}^{2}$ is an eigenvalue, the Jost solutions $e(\rho
_{k},x)$ and $g(\rho_{k},x)$ are necessarily linearly dependent).

The sets
\[
J^{\pm}=\left\{  s^{\pm}(\rho),\,\rho\in\mathbf{R};\,\lambda_{k},\,\alpha
_{k}^{\pm},\,k=\overline{1,N}\right\}
\]
are called the right and left scattering data, respectively.

To solve the direct scattering problem means to find $J^{+}$ or $J^{-}$ by a
given potential $q$ satisfying (\ref{cond q standard}). The inverse scattering
problem consists in recovering the potential $q$ from a given set of
scattering data $J^{+}$ or $J^{-}$.

\section{Inverse scattering transform method}

Here we briefly recall the inverse scattering transform method for solving the
Cauchy problem for the KdV equation%
\begin{equation}
u_{t}-6uu_{x}+u_{xxx}=0. \label{KdV}%
\end{equation}
Consider (\ref{KdV}) subject to the initial condition
\begin{equation}
u(x,0)=q(x),\quad x\in\left(  -\infty,\infty\right)  , \label{init cond}%
\end{equation}
where $q$ is assumed to be a given real valued function satisfying
(\ref{cond q standard}). This Cauchy problem is uniquely solvable. Then in
order to obtain the solution $u(x,t)$ at any prescribed instant $t$, according
to ISTM, the following steps should be performed.

1) Compute the set $J^{+}$ or $J^{-}$.

2) Apply the evolution law to the scattering data as follows%
\[
J^{\pm}(t)=\left\{  s^{\pm}\left(  \rho,t\right)  =s^{\pm}\left(  \rho\right)
e^{\pm8i\rho^{3}t},\,\rho\in\mathbb{R};\,\lambda_{k}(t)=\lambda_{k}=-\tau
_{k}^{2},\,\alpha_{k}^{\pm}(t)=\alpha_{k}^{\pm}e^{\pm8\tau_{k}^{3}t}\right\}
.
\]
That is, the eigenvalues $\lambda_{k}$ do not change for $t\geq0$ while to
compute the norming constants and the reflection coefficients for $t>0$ one
needs simply multiply their values corresponding to $t=0$ (or to $q(x)$) by a
corresponding exponential factor.

3) Solve the inverse scattering problem for $J^{+}(t)$ or $J^{-}(t)$. The
recovered potential is precisely $u(x,t)$.

\section{Representations for scattering
data\label{Sect Representations scattering data}}

Computation of the discrete spectral data (eigenvalues and norming constants)
requires considering purely imaginary values of $\rho$, such that $\rho=i\tau
$, $\tau>0$, while for computing reflection coefficients $s^{\pm}(\rho)$ one
needs to consider all real values of $\rho$. In terms of the parameter $z$
defined by (\ref{z}) this means that computation of discrete spectral data
should be performed on the interval $\left(  -1,1\right)  $ (when
$\lambda\rightarrow-\infty$ one has that $z\rightarrow-1$, and $\lambda=0$
corresponds to $z=1$), while computation of the reflection coefficients is
done on the unitary circle $z=e^{i\theta}$ with $\theta\in\left(  -\pi
,\pi\right)  $ ($\theta=\pm\pi$ corresponds to $\rho=\pm\infty$).

\subsection{Discrete spectral data}

Having obtained the coefficients $\left\{  a_{n}\right\}  $, $\left\{
b_{n}\right\}  $, $\left\{  c_{n}\right\}  $ and $\left\{  d_{n}\right\}  $ as
explained in Appendix, with the aid of the series representations
(\ref{e via z})-(\ref{g prime c}) it is easy to compute the square roots
$\rho_{k}$ of the eigenvalues as zeros of the function $a(\rho)$. We have
\[
-2i\rho a(\rho)=W\left[  e(\rho,0),g(\rho,0)\right]  =e(z)G(z)-E(z)g(z),
\]
where
\begin{equation}
e(z)=1+\left(  z+1\right)  \sum_{n=0}^{\infty}\left(  -1\right)  ^{n}%
z^{n}a_{n}(0), \label{e(z)}%
\end{equation}%
\begin{equation}
g(z)=1+\left(  z+1\right)  \sum_{n=0}^{\infty}\left(  -1\right)  ^{n}%
z^{n}b_{n}(0), \label{g(z)}%
\end{equation}%
\begin{equation}
E(z)=\frac{z-1}{2(z+1)}-\frac{1}{2}\int_{0}^{\infty}q(t)dt+(z+1)\sum
_{n=0}^{\infty}\left(  -1\right)  ^{n}z^{n}d_{n}(0) \label{E(z)}%
\end{equation}
and
\begin{equation}
G(z)=-\frac{z-1}{2(z+1)}+\frac{1}{2}\int_{-\infty}^{0}q(t)dt+(z+1)\sum
_{n=0}^{\infty}\left(  -1\right)  ^{n}z^{n}c_{n}(0). \label{G(z)}%
\end{equation}
Thus, computation of the eigenvalues reduces to computation of zeros $z_{k}$
of the function
\[
\Phi(z)=e(z)G(z)-E(z)g(z)
\]
on the interval $(-1,1)$, and
\[
\lambda_{k}=-\left(  \frac{z_{k}-1}{2\left(  z_{k}+1\right)  }\right)  ^{2}.
\]

For computing the norming constants $\alpha_{k}^{\pm}$ it is convenient to use
(\ref{akpha_n}). Here again, to find $d_{k}$ one may compute $e\left(
\rho_{k},0\right)  $ and $g\left(  \rho_{k},0\right)  $ by (\ref{e via z}) and
(\ref{g via z}), and then $d_{k}=g\left(  z_{k}\right)  /e\left(
z_{k}\right)  $.

To obtain $a^{\prime}(\rho_{k})$ which is required by (\ref{akpha_n}), we
notice that
\begin{equation*}
a^{\prime}(\rho)    =-\frac{a(\rho)}{\rho}-\frac{1}{2i\rho}W^{\prime}\left[
e(\rho,0),g(\rho,0)\right]  =2i\left(  \frac{z+1}{z-1}\right)  ^{2}\Phi(z)-\frac{z+1}{z-1}\Phi^{\prime
}(z)\frac{dz}{d\rho}.
\end{equation*}
Since $\frac{dz}{d\rho}=i\left(  z+1\right)  ^{2}$, we obtain
\[
a^{\prime}(\rho)=2i\left(  \frac{z+1}{z-1}\right)  ^{2}\Phi(z)-i\frac{\left(
z+1\right)  ^{3}}{z-1}\Phi^{\prime}(z)
\]
where $\Phi^{\prime}(z)$ is easily calculated by using the representations
(\ref{e(z)})-(\ref{G(z)}).

\subsection{Reflection coefficients}

Computation of the reflection coefficients is performed according to formula
(\ref{reflection}), from which by taking into account (\ref{Jost conjugation})
we obtain%
\[
s^{+}(\rho)=\frac{W\left[  \overline{e(\rho,0)},g(\rho,0)\right]  }{W\left[
e(\rho,0),g(\rho,0)\right]  }%
\]
and
\[
s^{-}(\rho)=-\frac{W\left[  e(\rho,0),\overline{g(\rho,0)}\right]  }{W\left[
e(\rho,0),g(\rho,0)\right]  }.
\]
In terms of the functions (\ref{e(z)})-(\ref{G(z)}) we have
\[
s^{+}(\rho)=\frac{e(\overline{z})G(z)-E(\overline{z})g(z)}{e(z)G(z)-E(z)g(z)},\qquad \text{where $z$ is given by \eqref{z},}
\]
and
\[
s^{-}(\rho)=-\frac{e(z)G(\overline{z})-E(z)g(\overline{z})}{e(z)G(z)-E(z)g(z)}%
.
\]
These expressions have to be computed for $z$ running along the unitary
circle, namely for $z=e^{i\theta}$, $\theta\in(-\pi,\pi)$ which corresponds to
$\rho$ running along the real axis from $-\infty$ to $+\infty$.

\section{Inverse scattering}

\subsection{System of linear algebraic equations for the Fourier-Laguerre
coefficients}

The transmutation operator kernels $A$ and $B$ satisfy the corresponding
Marchenko equations, often called Gelfand-Levitan-Marchenko equations (see,
e.g., \cite{Chadan}, \cite{Levitan}, \cite{Yurko2007})%
\begin{equation}
F^{+}(x+y)+A(x,y)+\int_{x}^{\infty}A(x,t)F^{+}(t+y)dt=0,\quad y>x \label{GLMA}%
\end{equation}
and
\begin{equation}
F^{-}(x+y)+B(x,y)+\int_{-\infty}^{x}B(x,t)F^{-}(t+y)dt=0,\quad y<x
\label{GLMB}%
\end{equation}
where
\begin{equation}
F^{\pm}(x)=\sum_{k=1}^{N}\alpha_{k}^{\pm}e^{\mp\tau_{k}x}+\frac{1}{2\pi}%
\int_{-\infty}^{\infty}s^{\pm}\left(  \rho\right)  e^{\pm i\rho x}d\rho.
\label{F}%
\end{equation}
Equations (\ref{GLMA}) and (\ref{GLMB}) are uniquely solvable. However, the
method proposed in the present work does not involve their solution. In fact,
we use equations (\ref{GLMA}) and (\ref{GLMB}) in order to obtain a system of
linear algebraic equations for the coefficients $\left\{  a_{n}\right\}
_{n=0}^{\infty}$ and $\left\{  b_{n}\right\}  _{n=0}^{\infty}$, respectively,
which allows us to compute $a_{0}$ and $b_{0}$ and hence to recover $q$ by
(\ref{q=a}) and (\ref{q=b}).

Let us introduce the following notation%
\begin{equation}
A_{mn}(x):=(-1)^{m+n}\left(  \sum_{k=1}^{N}\alpha_{k}^{+}e^{-2\tau_{k}x}%
\frac{\left(  \frac{1}{2}-\tau_{k}\right)  ^{m+n}}{\left(  \frac{1}{2}%
+\tau_{k}\right)  ^{m+n+2}}+\frac{1}{2\pi}\int_{-\infty}^{\infty}s^{+}\left(
\rho\right)  e^{2i\rho x}\frac{\left(  \frac{1}{2}+i\rho\right)  ^{m+n}%
}{\left(  \frac{1}{2}-i\rho\right)  ^{m+n+2}}d\rho\right)  , \label{alpha mn}%
\end{equation}%
\begin{equation}
r_{m}(x):=(-1)^{m+1}\left(  \sum_{k=1}^{N}\alpha_{k}^{+}e^{-2\tau_{k}x}%
\frac{\left(  \frac{1}{2}-\tau_{k}\right)  ^{m}}{\left(  \frac{1}{2}+\tau
_{k}\right)  ^{m+1}}+\frac{1}{2\pi}\int_{-\infty}^{\infty}s^{+}\left(
\rho\right)  e^{2i\rho x}\frac{\left(  \frac{1}{2}+i\rho\right)  ^{m}}{\left(
\frac{1}{2}-i\rho\right)  ^{m+1}}d\rho\right)  , \label{r m}%
\end{equation}%
\begin{equation}
B_{mn}(x):=(-1)^{m+n}\left(  \sum_{k=1}^{N}\alpha_{k}^{-}e^{2\tau_{k}x}%
\frac{\left(  \frac{1}{2}-\tau_{k}\right)  ^{m+n}}{\left(  \frac{1}{2}%
+\tau_{k}\right)  ^{m+n+2}}+\frac{1}{2\pi}\int_{-\infty}^{\infty}s^{-}\left(
\rho\right)  e^{-2i\rho x}\frac{\left(  \frac{1}{2}+i\rho\right)  ^{m+n}%
}{\left(  \frac{1}{2}-i\rho\right)  ^{m+n+2}}d\rho\right)  , \label{beta mn}%
\end{equation}%
\begin{equation}
s_{m}(x):=(-1)^{m+1}\left(  \sum_{k=1}^{N}\alpha_{k}^{-}e^{2\tau_{k}x}%
\frac{\left(  \frac{1}{2}-\tau_{k}\right)  ^{m}}{\left(  \frac{1}{2}+\tau
_{k}\right)  ^{m+1}}+\frac{1}{2\pi}\int_{-\infty}^{\infty}s^{-}\left(
\rho\right)  e^{-2i\rho x}\frac{\left(  \frac{1}{2}+i\rho\right)  ^{m}%
}{\left(  \frac{1}{2}-i\rho\right)  ^{m+1}}d\rho\right)  . \label{s m}%
\end{equation}

\begin{theorem}
\cite{Kr2019MMAS InverseAxis}, \cite{KrBook2020} The coefficients $\left\{
a_{n}\right\}  _{n=0}^{\infty}$ and $\left\{  b_{n}\right\}  _{n=0}^{\infty}$
from (\ref{A series}) and (\ref{B series}), respectively, satisfy the
following systems of linear algebraic equations
\begin{equation}
a_{m}(x)+\sum_{n=0}^{\infty}a_{n}(x)A_{mn}(x)=r_{m}(x),\quad m=0,1,\ldots
\label{syst A}%
\end{equation}
and
\begin{equation}
b_{m}(x)+\sum_{n=0}^{\infty}b_{n}(x)B_{mn}(x)=s_{m}(x),\quad m=0,1,\ldots.
\label{syst B}%
\end{equation}

For every $m$ the series in (\ref{syst A}) and (\ref{syst B}) converge pointwise.
\end{theorem}

The system \eqref{syst A}--\eqref{syst B} is well suited for approximate numerical solution by the finite section method, i.e., considering $m,n\le N$. Similarly to \cite{KShT2020}, \cite{KT2020} one can check that the truncated system coincides with that obtained by applying the Bubnov-Galerkin process to the integral equation \eqref{GLMA} with respect to the orthonormal system of Laguerre functions. Hence one obtains that the truncated system is uniquely solvable for all sufficiently large $N$, the approximate solution converges to the exact one as $N\to\infty$, the condition numbers of the coefficient matrices are uniformly bounded with respect to $N$, and the solution is stable with respect to small errors in the coefficients, see \cite{Mihlin}. We leave the details to the reader. However, we would like to point out that we do not need the whole solution of the system \eqref{syst A}--\eqref{syst B}, only the first component $a_0$ is sufficient to recover the potential.

In the following sections we present a different approach to the study of the system \eqref{syst A}--\eqref{syst B} based on the Hankel equation and Hankel operators.

\subsection{Derivation of systems (\ref{syst A}) and (\ref{syst B}) from
Hankel equation}

Systems (\ref{syst A}) and (\ref{syst B}) are derived from the
Gelfand-Levitan-Marchenko equations. On the other hand the inverse scattering
problem can be reduced to another integral equation (see \cite{GR2015SIAM})
which has the form
\begin{equation}
\left(  \left(  I+\mathbf{H}(\varphi)\right)  y\right)  \left(  \rho\right)
=-\mathbf{H}(\varphi)\left(  1\right)  \left(  \rho\right)
\label{Equation Hankel}%
\end{equation}
where
\begin{gather*}
y\left(  \rho\right)  =e^{-i\rho x}e\left(  \rho,x\right)  -1,\\
\left(  \mathbf{H}(\varphi)y\right)  \left(  \rho\right)  =\frac{1}{2\pi
i}\int_{-\infty}^{\infty}\frac{\varphi\left(  \tau\right)  y\left(
\tau\right)  }{\tau+\rho}d\tau,\quad\rho\in\mathbb{R}%
\end{gather*}
is a Hankel operator. Here the integral is understood in the sense of the
limit value%
\[
\int_{-\infty}^{\infty}\frac{\varphi\left(  \tau\right)  y\left(  \tau\right)
}{\tau+\rho}d\tau=\lim_{\varepsilon\rightarrow+0}\int_{-\infty}^{\infty}%
\frac{\varphi\left(  \tau\right)  y\left(  \tau\right)  }{\tau+\left(
\rho+i\varepsilon\right)  }d\tau.
\]
The symbol of the Hankel operator has the form%
\[
\varphi\left(  \tau\right)  =R(\tau)e^{8it\tau^{3}+2ix\tau}%
\]
where $R(\tau)$ is the reflection coefficient ($R(\tau)\equiv s^{-}(\tau)$).

Let us derive system (\ref{syst A}) from (\ref{Equation Hankel}). We have
\begin{equation}
y(\tau)=\left(  z(\tau)+1\right)  \sum_{n=0}^{\infty}\left(  -1\right)
^{n}z^{n}(\tau)a_{n}(x)=\frac{1}{\frac{1}{2}-i\tau}\sum_{n=0}^{\infty}\left(
-1\right)  ^{n}z^{n}(\tau)a_{n}(x). \label{Jost representation}%
\end{equation}
Substitution of (\ref{Jost representation}) into (\ref{Equation Hankel}) gives%
\begin{equation}
\sum_{n=0}^{\infty}\left(  -1\right)  ^{n}\frac{z^{n}(\rho)a_{n}(x)}%
{\frac{1}{2}-i\rho}+\sum_{n=0}^{\infty}\left(  -1\right)  ^{n}\frac{a_{n}%
(x)}{2\pi i}\int_{-\infty}^{\infty}\frac{\varphi\left(  \tau\right)
z^{n}(\tau)}{\left(  \frac{1}{2}-i\tau\right)  \left(  \tau+\rho\right)
}d\tau  =-\frac{1}{2\pi i}\int_{-\infty}^{\infty}\frac{\varphi\left(  \tau\right)
}{\left(  \tau+\rho\right)  }d\tau. \label{Equation Hankel  1}%
\end{equation}
Let us make use of the orthogonality of the system $\left\{  z^{n}\right\}
_{-\infty}^{\infty}$ on the unitary circle%
\[
\delta_{n,m}=\frac{1}{2\pi i}\int_{\mathbb{T}}z^{n}\overline{z}^{m}\frac
{dz}{z}=\left\{
\begin{array}
[c]{c}%
1,\quad n=m,\\
0,\quad n\neq m.
\end{array}
\right.
\]

Changing the variable
\[
z=z(\tau)=\frac{\frac{1}{2}+i\tau}{\frac{1}{2}-i\tau}:\quad\mathbb{R}%
\rightarrow\mathbb{T},
\]%
\[
dz=\frac{id\tau}{\left(  \frac{1}{2}-i\tau\right)  ^{2}}%
\]
we obtain%
\begin{equation*}
\delta_{n,m}    =\frac{1}{2\pi}\int_{-\infty}^{\infty}\frac{z^{n}%
(\tau)\overline{z}^{m}(\tau)}{\left(  \frac{1}{2}-i\tau\right)  ^{2}}%
\frac{d\tau}{z(\tau)}  =\frac{1}{2\pi}\int_{-\infty}^{\infty}\frac{z^{n-(m+1)}(\tau)}{\left(
\frac{1}{2}-i\tau\right)  ^{2}}d\tau.
\end{equation*}

Multiplying (\ref{Equation Hankel 1}) by
\[
\frac{(-1)^{m}}{2\pi}\frac{z^{-(m+1)}(\rho)}{\frac{1}{2}-i\rho}%
\]
and integrating over $\mathbb{R}$ leads to the equality%
\begin{align}
&  a_{m}(x)+\frac{(-1)^{m}}{2\pi i}\sum_{n=0}^{\infty}\left(  -1\right)
^{n}a_{n}(x)\int_{-\infty}^{\infty}\frac{\varphi\left(  \tau\right)
z^{n}(\tau)d\tau}{\left(  \frac{1}{2}-i\tau\right)  }\left(  \frac{1}{2\pi
}\int_{-\infty}^{\infty}\frac{z^{-(m+1)}(\rho)}{\left(  \frac{1}{2}%
-i\rho\right)  \left(  \tau+\rho\right)  }d\rho\right) \nonumber\\
&  =-\frac{(-1)^{m}}{2\pi i}\int_{-\infty}^{\infty}\varphi\left(  \tau\right)
\frac{d\tau}{2\pi}\left(  \int_{-\infty}^{\infty}\frac{z^{-(m+1)}(\rho
)}{\left(  \frac{1}{2}-i\rho\right)  \left(  \tau+\rho\right)  }d\rho\right)
,\quad m=0,1,\ldots. \label{Equation Hankel  2}%
\end{align}

The interior integral can be calculated with the aid of the residue theory by
considering $\mathbb{R}$ as a closed contour enclosing the lower half-plane.
The function
\[
\frac{z^{-(m+1)}(\rho)}{\frac{1}{2}-i\rho}%
\]
has no pole in the lower half-plane. Hence
\begin{equation*}
\lim_{\varepsilon\rightarrow0+}\frac{1}{2\pi}\int_{-\infty}^{\infty}%
\frac{z^{-(m+1)}(\rho)}{\left(  \frac{1}{2}-i\rho\right)  \left(  \tau
+\rho+i\varepsilon\right)  }d\rho   =\lim_{\varepsilon\rightarrow0+}\left(
\frac{-iz^{-(m+1)}\left(  -\tau-i\varepsilon\right)  }{\frac{1}{2}-i\left(
-\tau-i\varepsilon\right)  }\right)  =-i\frac{z^{-(m+1)}\left(  -\tau\right)  }{\frac{1}{2}+i\tau}.
\end{equation*}
Thus (\ref{Equation Hankel 2}) takes the form%
\begin{align*}
&  a_{m}(x)+\sum_{n=0}^{\infty}\frac{(-1)^{n+m+1}}{2\pi}a_{n}(x)\int_{-\infty
}^{\infty}\frac{\varphi\left(  \tau\right)  z^{n}(\tau)z^{-(m+1)}\left(
-\tau\right)  d\tau}{\left(  \frac{1}{2}-i\tau\right)  \left(  \frac{1}%
{2}+i\tau\right)  }\\
&  =\frac{(-1)^{m}}{2\pi}\int_{-\infty}^{\infty}\frac{\varphi\left(
\tau\right)  z^{-(m+1)}(-\tau)d\tau}{\frac{1}{2}+i\tau},\quad m=0,1,\ldots.
\end{align*}
Since $z(-\tau)=z^{-1}(\tau)$, we obtain%
\begin{align}
&  a_{m}(x)+\sum_{n=0}^{\infty}\frac{(-1)^{n+m+1}}{2\pi}a_{n}(x)\int_{-\infty
}^{\infty}\frac{\varphi\left(  \tau\right)  z^{n+m+1}(\tau)d\tau}{\left(
\frac{1}{2}-i\tau\right)  \left(  \frac{1}{2}+i\tau\right)  }\nonumber\\
&  =\frac{(-1)^{m}}{2\pi}\int_{-\infty}^{\infty}\frac{\varphi\left(
\tau\right)  z^{m+1}(\tau)d\tau}{\frac{1}{2}+i\tau},\quad m=0,1,\ldots.
\label{Equation Hankel  3}%
\end{align}

This system coincides with (\ref{syst A}). Clearly, instead of
(\ref{Equation Hankel 1}) we can consider an equation corresponding to the
function $g(\rho,x)$ and obtain (\ref{syst B}).

\begin{remark}
System (\ref{Equation Hankel 3}) represents an infinite system of linear
algebraic equations generated by a Hankel matrix, entries of which depend on
the sum of indices $(m+n)$. This is not accidental. In the work
\cite{GR2015SIAM} and in some other it is proved that the operator on the left
hand side of (\ref{Equation Hankel}) is invertible in the space $L_{2}%
^{+}\left(  \mathbb{R}\right)  $. Hence $y\left(  \rho\right)  \in L_{2}%
^{+}\left(  \mathbb{R}\right)  $. This implies that $y\left(  \rho\right)  $
admits a series representation of the form
\begin{equation}
y\left(  \rho\right)  :=y\left(  \rho,x\right)  =\sum_{n=0}^{\infty}%
\frac{a_{n}(x)z^{n}\left(  \rho\right)  }{\frac{1}{2}-i\rho} \label{repr y}%
\end{equation}
where
\[
\sum_{n=0}^{\infty}\left\vert a_{n}(x)\right\vert ^{2}<\infty.
\]
Indeed, $f(z)\in L_{2}^{+}\left(  \mathbb{T}\right)  $ is equivalent to the
equality
\[
f(z)=\sum_{n=0}^{\infty}f_{n}z^{n}%
\]
where $\sum_{n=0}^{\infty}\left\vert f_{n}\right\vert ^{2}<\infty$. The
operator of the change of the variable
\[
\left(  Uf\right)  \left(  \rho\right)  =\frac{1}{\frac{1}{2}-i\rho}f(z\left(
\rho\right)  ):\quad L_{2}\left(  \mathbb{T}\right)  \rightarrow L_{2}\left(
\mathbb{R}\right)
\]
is an isometry. Thus we obtain the representation (\ref{repr y}). It in turn
coincides with (\ref{Jost representation}) up to the factor $(-1)^{n}$, but is
obtained independently of Levin's representation. Thus, system
(\ref{Equation Hankel 3}) can be written in the form%
\[
\left(  I+K\right)  X=f
\]
where $K$ is a compact operator, and the operator $I+K$ is invertible.
\end{remark}

\section{Convergence of the finite section
method\label{Sect Convergence of method of reduction}}

Consider the Hankel operator defining equation (\ref{Equation Hankel}). It can
be written in the form%
\[
\left(  \mathbf{H}(\varphi)y\right)  \left(  \rho\right)  =\left(
JP^{-}\varphi y\right)  \left(  \rho\right)
\]
where the analytic projection operator $P^{-}$ has the form%
\[
\left(  P^{-}f\right)  \left(  \rho\right)  =\lim_{\varepsilon\rightarrow
0+}\frac{1}{2\pi i}\int_{-\infty}^{\infty}\frac{f\left(  \tau\right)  d\tau
}{\tau-\left(  \rho-i\varepsilon\right)  }%
\]
and $\left(  Jf\right)  \left(  \rho\right)  =f\left(  -\rho\right)  $ is the
reflection operator.

Obviously, if $h^{+}\left(  \rho\right)  \in H_{\infty}^{+}\left(
\mathbb{R}\right)  $, where $H_{\infty}^{+}\left(  \mathbb{R}\right)  $ is the
set of all bounded analytic functions in the upper half-plane, then
\[
\mathbf{H}(\varphi+h^{+})=\mathbf{H}(\varphi).
\]
Hence instead of the symbol $\varphi$ we can introduce a modified symbol
$\varphi+h^{+}$. The function $\varphi\left(  \rho\right)  $ has the form
\begin{equation}
\varphi\left(  \rho\right)  =R(\rho)e^{8it\rho^{3}+2ix\rho}.
\label{symbol phi}%
\end{equation}
At first sight the highly oscillating factor deteriorates the smoothness
properties of $\varphi\left(  \rho\right)  $, however the smoothness
properties of $\varphi+h^{+}$ improve (in comparison with $R(\rho)$). The
smoothness of the Hankel operator symbol is important since the higher it is,
the faster is the decay of the elements of the Hankel matrix.

Together with the symbol (\ref{symbol phi}) we consider its derivatives with
respect to $x$:
\begin{equation}
\frac{\partial^{j}}{\partial x^{j}}\varphi\left(  \rho,x\right)  =\left(
2i\rho\right)  ^{j}R(\rho)e^{8it\rho^{3}+2ix\rho}. \label{symbol der}%
\end{equation}

In \cite{GR2020Proc} it was proved that under the condition
\begin{equation}
q(x)\in L_{1}(\mathbb{R},\left(  1+\left\vert x\right\vert \right)  ^{\alpha
}),\quad\alpha\geq1 \label{cond q}%
\end{equation}
the reflection coefficient admits the representation%
\begin{equation}
R(\rho)=T_{+}(\rho)G_{-}(\rho) \label{R factorized}%
\end{equation}
with
\begin{equation}
T_{+}(\rho)\in H_{\infty}^{+}(\mathbb{R}) \label{T+}%
\end{equation}
and
\begin{equation}
G_{-}(\rho)=\frac{1}{2i\rho}\int_{0}^{\infty}q(s)e^{-2is\rho}ds+\frac
{1}{\left(  2i\rho\right)  ^{2}}\int_{0}^{\infty}Q^{\prime}(s)e^{-2is\rho}ds
\label{G-}%
\end{equation}
where $Q(s)$ is an absolutely continuous function satisfying the inequality%
\begin{equation}
\left\vert Q^{\prime}(s)\right\vert \leq C_{1}\left\vert q(s)\right\vert
+C_{2}\int_{s}^{\infty}\left\vert q(u)\right\vert du \label{Qprime}%
\end{equation}
where $C_{1}$ and $C_{2}$ are independent of $s$. Note that this inequality
implies the inclusion $Q^{\prime}\in L_{1}(\mathbb{R},\left(  1+\left\vert
x\right\vert \right)  ^{\alpha-1})$.

Denote%
\begin{equation}
S\left(  \rho,x\right)  =8t\rho^{3}+2x\rho. \label{s}%
\end{equation}

Let us show first that the finite section method is applicable to system
(\ref{Equation Hankel 3}) under a quite general assumption when the symbol
$\varphi\left(  \rho\right)  $ is a continuous on the closed real axis
$\overset{\cdot}{\mathbb{R}}$ function. The function of the form
(\ref{symbol phi}) fulfils this condition since $R(\pm\infty)=0$.

Let us introduce two families of projection operators on the space $L_{2}%
^{+}\left(  \mathbb{R}\right)  :=P^{+}\left(  L_{2}\left(  \mathbb{R}\right)
\right)  $.

Let
\[
f\left(  \rho\right)  =\sum_{k=0}^{\infty}\frac{f_{k}}{\frac{1}{2}-i\rho}%
z^{k}\left(  \rho\right)  ,\quad\sum_{k=0}^{\infty}\left\vert f_{k}\right\vert
^{2}<\infty
\]
belong to the class $L_{2}^{+}\left(  \mathbb{R}\right)  $. We recall that
$z\left(  \rho\right)  =\frac{\frac{1}{2}+i\rho}{\frac{1}{2}-i\rho}$. Then
\[
\left(  P_{n}f\right)  \left(  \rho\right)  :=\sum_{k=0}^{n-1}\frac{f_{k}%
}{\frac{1}{2}-i\rho}z^{k}\left(  \rho\right)  ,\quad\left(  Q_{n}f\right)
\left(  \rho\right)  :=\sum_{k=n}^{\infty}\frac{f_{k}}{\frac{1}{2}-i\rho}%
z^{k}\left(  \rho\right)  .
\]

\begin{theorem}
\label{Th reduction in L2}Let $\varphi\left(  \rho\right)  \in C\left(
\overset{\cdot}{\mathbb{R}}\right)  $ and the operator $\left(  I+\mathbf{H}%
(\varphi)\right)  $ be invertible in the space $L_{2}^{+}\left(
\mathbb{R}\right)  $. Then the finite section method is applicable to system
(\ref{Equation Hankel 3}) in the same space.
\end{theorem}

\begin{proof}
Instead of the system (\ref{Equation Hankel 3}) consider its reduced analogue%
\begin{align*}
&  a_{m}(x)+\sum_{k=0}^{n-1}\frac{(-1)^{k+m+1}}{2\pi}a_{k}(x)\int_{-\infty
}^{\infty}\frac{\varphi\left(  \tau\right)  z^{k+m+1}(\tau)d\tau}{\frac{1}%
{4}+\tau^{2}}\\
&  =\frac{(-1)^{m}}{2\pi}\int_{-\infty}^{\infty}\frac{\varphi\left(
\tau\right)  z^{m+1}(\tau)d\tau}{\frac{1}{2}+i\tau},\quad m=0,1,\ldots,n-1.
\end{align*}
This system can be written in the form%
\begin{equation}
\left(  \left(  P_{n}+P_{n}\mathbf{H}(\varphi)\right)  P_{n}y\right)  \left(
\rho\right)  =-P_{n}\mathbf{H}(\varphi)\left(  1\right)  \left(  \rho\right)
. \label{reduced system in operators}%
\end{equation}
Denote $A^{-1}:=\left(  I+\mathbf{H}(\varphi)\right)  ^{-1}$ and consider the
product
\begin{align}
P_{n}\left(  I+\mathbf{H}(\varphi)\right)  P_{n}\cdot P_{n}A^{-1}P_{n}  &
=P_{n}\left(  I+\mathbf{H}(\varphi)\right)  A^{-1}P_{n}-P_{n}\left(
I+\mathbf{H}(\varphi)\right)  Q_{n}A^{-1}P_{n}\nonumber\\
&  =P_{n}-P_{n}\mathbf{H}(\varphi)Q_{n}A^{-1}P_{n}. \label{operators eq}%
\end{align}
The operator $Q_{n}A^{-1}P_{n}$ tends to zero in a strong sense when
$n\rightarrow\infty$. Since the operator $\mathbf{H}(\varphi)$ is compact
(see, e.g., \cite[p. 77]{BottcherSilbermann2006}), the operator $P_{n}%
\mathbf{H}(\varphi)Q_{n}A^{-1}P_{n}$ tends to zero in the operator sense.
Thus, for a large enough $n$ the operator on the right-hand side of
(\ref{operators eq}) is invertible. Hence the operator $P_{n}\left(
I+\mathbf{H}(\varphi)\right)  P_{n}$ is invertible in the space $L_{2}%
^{+(n)}\left(  \mathbb{R}\right)  :=P_{n}L_{2}^{+}\left(  \mathbb{R}\right)
$, and, moreover, the norms of the inverse operators are uniformly bounded.
It remains to show that the solution of (\ref{reduced system in operators})
tends to the solution of (\ref{Equation Hankel 3}) when $n\rightarrow\infty$.
Indeed, let $f\in L_{2}^{+}\left(  \mathbb{R}\right)  $. Consider%
\begin{align*}
R_{n}f  &  :=\left(  \left(  I+\mathbf{H}(\varphi)\right)  ^{-1}-\left(
P_{n}\left(  I+\mathbf{H}(\varphi)\right)  P_{n}\right)  ^{-1}\right)  f\\
&  =-\left(  P_{n}+P_{n}\mathbf{H}(\varphi)P_{n}\right)  ^{-1}\left(  \left(
I+\mathbf{H}(\varphi)\right)  Q_{n}+Q_{n}\mathbf{H}(\varphi)P_{n}\right)
\left(  I+\mathbf{H}(\varphi)\right)  ^{-1}f+Q_{n}\left(  I+\mathbf{H}%
(\varphi)\right)  ^{-1}f.
\end{align*}
Thus,
\[
\left\Vert R_{n}f\right\Vert _{L_{2}}\leq M_{n}M\left\Vert Q_{n}%
\psi\right\Vert _{L_{2}}+M\left\Vert Q_{n}\psi_{n}\right\Vert _{L_{2}%
}+\left\Vert Q_{n}\psi\right\Vert _{L_{2}}%
\]
where $M_{n}=\left\Vert \left(  P_{n}\left(  I+\mathbf{H}(\varphi)\right)
P_{n}\right)  ^{-1}\right\Vert _{L_{2}}$, $M=\left\Vert I+\mathbf{H}%
(\varphi)\right\Vert _{L_{2}}$, $\psi_{n}=\mathbf{H}(\varphi)P_{n}\left(
I+\mathbf{H}(\varphi)\right)  ^{-1}f$ and $\psi=\left(  I+\mathbf{H}%
(\varphi)\right)  ^{-1}f$. Since $M_{n}$ is uniformly bounded with respect to
$n$, and the sequence of functions $\left\{  \psi_{n}\right\}  $ converges in
the $L_{2}$-norm, we obtain that
\[
\lim_{n\rightarrow\infty}\left\Vert R_{n}f\right\Vert _{L_{2}}=0.
\]
\end{proof}

\begin{remark}
In fact for the proof of Theorem \ref{Th reduction in L2} it is sufficient to
notice that the case under consideration is covered by Proposition 2.4 from
\cite{BottcherSilbermann1999}, however this would require some additional
notations and definitions.
\end{remark}

Thus, we proved the convergence of the finite section method for the system
(\ref{Equation Hankel 3}) in the space $l_{2}\left(  \mathbb{Z}\right)  $. For
the proof of a stronger convergence it is necessary to use some deeper results
concerning properties of the operator $\mathbf{H}(\varphi)$. In
\cite{GR2015SIAM}, \cite{GR2018MathNotes} it was shown that under the
assumption (\ref{cond q}) the operator $\mathbf{H}(\varphi)$ is nuclear
together with the operators
\[
\frac{\partial^{j}}{\partial x^{j}}\mathbf{H}(\varphi),\quad j=1,2,\ldots
,\left\lfloor 2\alpha\right\rfloor .
\]
We recall that the operator $K$ is nuclear if its singular numbers satisfy the
condition
\[
\sum_{j=1}^{\infty}\left\vert s_{j}(K)\right\vert <\infty.
\]

In the case $\alpha=1$ the first and the second derivatives of the operator
$\mathbf{H}(\varphi)$ are nuclear operators. The set of nuclear operators is
denoted by $G_{1}$.

According to system (\ref{Equation Hankel 3}) the matrix of the operator
$\mathbf{H}(\varphi)$ in the basis
\[
e_{k}(\rho)=\frac{z^{k}(\rho)}{\frac{1}{2}-i\rho},\quad k=0,1,\ldots
\]
has the form
\begin{equation}
\widehat{\mathbf{H}}(\varphi)=\left\{  h_{k+m}\right\}  _{k,m=0}^{\infty
},\quad h_{p}=\frac{\left(  -1\right)  ^{p+1}}{2\pi}\int_{-\infty}^{\infty
}\frac{\varphi\left(  \tau\right)  z^{p+1}\left(  \tau\right)  d\tau}{\frac
{1}{4}+\tau^{2}}. \label{matrix H}%
\end{equation}
As it is known, the trace of the matrix is invariant with respect to the
change of the basis. If the operator is nuclear the sum of the absolute values
of the diagonal elements is finite (see, e.g., \cite[p. 267]{Conway}). In our
case this means that
\[
\sum_{m=0}^{\infty}\left\vert h_{2m}\right\vert <\infty.
\]
Consider the symbol
\[
\varphi_{1}\left(  \rho\right)  =z\left(  \rho\right)  \varphi\left(
\rho\right)  .
\]
It is easy to see that
\[
\mathbf{H}(\varphi_{1})=\mathbf{H}(\varphi)\left(  P^{+}zP^{+}\right)  .
\]
Thus, $\mathbf{H}(\varphi_{1})$ also belongs to $G_{1}$. The matrix of the
operator $\mathbf{H}(\varphi_{1})$ has the form
\[
\widehat{\mathbf{H}}(\varphi_{1})=\left\{  h_{k+m+1}\right\}  _{k,m=0}%
^{\infty}.
\]
Thus we have
\[
\sum_{m=0}^{\infty}\left\vert h_{2m+1}\right\vert <\infty
\]
and hence
\[
\sum_{m=0}^{\infty}\left\vert h_{m}\right\vert <\infty.
\]

Let us introduce the space $l_{1}\left(  \mathbb{R}\right)  $ of functions
admitting the representation
\[
f\left(  \rho\right)  =\frac{1}{\frac{1}{2}-i\rho}\sum_{k=0}^{\infty}%
f_{k}z^{k}(\rho),\quad\sum_{k=0}^{\infty}\left\vert f_{k}\right\vert <\infty.
\]
Then the following statement is valid.

\begin{theorem}
\label{Th H is compact in l1}Let the function $\varphi\left(  \rho\right)  $
of the form (\ref{symbol phi}) satisfy the conditions (\ref{cond q}%
)-(\ref{Qprime}). Then the operator $\mathbf{H}(\varphi)$ is compact in the
space $l_{1}\left(  \mathbb{R}\right)  $ and the operator $I+\mathbf{H}%
(\varphi)$ is bounded and invertible in this space.
\end{theorem}

\begin{Proof}
Let us show that $\mathbf{H}(\varphi)$ can be represented as a limit with
respect to the operator norm of the family of finite-dimensional operators
$\left\{  K_{n}\right\}  _{n=1}^{\infty}$. Indeed, let $f\left(  \rho\right)
\in l_{1}\left(  \mathbb{R}\right)  $. Then
\begin{equation*}
u_{n}\left(  \rho\right)     :=\left(  \mathbf{H}(\varphi)-P_{n}%
\mathbf{H}(\varphi)\right)  f\left(  \rho\right)   =
\begin{cases}
0,& m=0,1,\ldots,n-1\\
\sum_{k=0}^{\infty}h_{k+m}f_{k}\frac{z^{k}(\rho)}{\frac{1}{2}-i\rho},&
m=n,n+1,\ldots.
\end{cases}
\end{equation*}
Obviously,
\[
\left\Vert u_{n}\left(  \rho\right)  \right\Vert _{l_{1}}\leq\left(
\sum_{k=n}^{\infty}\left\vert h_{k}\right\vert \right)  \left\Vert
f\right\Vert _{l_{1}}.
\]
Thus,
\[
\lim_{n\rightarrow\infty}\left\Vert u_{n}\left(  \rho\right)  \right\Vert
_{l_{1}}=0,
\]
and the operator $\mathbf{H}(\varphi)$ is compact in the space $l_{1}\left(
\mathbb{R}\right)  $. This implies that the operator $I+\mathbf{H}(\varphi)$
is a Fredholm operator with the index $0$. On the other hand we already know
that the operator $I+\mathbf{H}(\varphi)$ is invertible in the space
$L_{2}^{+}\left(  \mathbb{R}\right)  $. Hence
\[
\left.  \ker\left(  I+\mathbf{H}(\varphi)\right)  \right\vert _{l_{1}\left(
\mathbb{R}\right)  }=\left\{  0\right\}  ,
\]
since $l_{1}\left(  \mathbb{R}\right)  \subset L_{2}^{+}\left(  \mathbb{R}%
\right)  $. Consequently, $\left(  I+\mathbf{H}(\varphi)\right)  $ is
invertible in $l_{1}\left(  \mathbb{R}\right)  $.
\end{Proof}

\begin{theorem}
\label{Th applicability of reduction in l1}Let the function $\varphi\left(
\rho\right)  $ of the form (\ref{symbol phi}) satisfy the conditions
(\ref{cond q})-(\ref{Qprime}). Then the system (\ref{Equation Hankel 3})
admits the finite section method in the space $l_{1}\left(  \mathbb{R}\right)
$.
\end{theorem}

The proof is analogous to that of Theorem \ref{Th reduction in L2}.

For considering the convergence in the stronger sense we make use of the
conditions (\ref{R factorized})--(\ref{Qprime}) for the potential of the kind
(\ref{cond q}). One can expect that the larger is $\alpha$ in (\ref{cond q})
the faster the finite section method applied to system
(\ref{Equation Hankel 3}) converges. Indeed, elements of the matrix
(\ref{matrix H}) are Fourier coefficients (this becomes obvious upon returning
to the unitary circle). It is well known that on a unitary circle the
existence of $l$-th derivative of a function implies the decay of its Fourier
coefficients as $o\left(  \left\vert \rho\right\vert ^{-l}\right)  $, and this
fact is proved by a direct integration by parts ($l$ times). In the case of
the real axis $\mathbb{R}$ the situation turns more complicated due to the
necessity of taking into account the behaviour at infinity. However the
technique developed in \cite{GR2015SIAM}, \cite{GR2018MathNotes} allows one to
prove the following results.

Let
\begin{equation}
\varphi_{0}\left(  \rho\right)  =G_{-}\left(  \rho\right)  e^{iS(\rho,x)}
\label{phi0}%
\end{equation}
($\varphi\left(  \rho\right)  =T_{+}\left(  \rho\right)  \varphi_{0}\left(
\rho\right)  $, the factor $T_{+}\left(  \rho\right)  $ will be taken into
account in Theorem \ref{Th asymptotics hp}), where $G_{-}\left(  \rho\right)
$ satisfies (\ref{G-}), (\ref{Qprime}) and $S(\rho,x)$ has the form (\ref{s}).
Denote%
\begin{equation}
\varphi_{0}^{-}\left(  \rho\right)  :=\frac{1}{2\pi i}\int_{-\infty}^{\infty
}\frac{\varphi_{0}\left(  \tau\right)  }{\tau-\rho}d\tau,\quad\rho
\in\mathbb{R}, \label{phi0-}%
\end{equation}
where the integral is understood in the sense of the limit value from the
lower half-plane.

\begin{theorem}
\label{Th estimate for hp}Let $\varphi_{0}^{-}\left(  \rho\right)  $ have the
form (\ref{phi0-}). Then its Fourier coefficients, the numbers%
\begin{equation}
h_{p}^{\left(  0\right)  }=\frac{1}{2\pi}\int_{-\infty}^{\infty}\varphi
_{0}^{-}\left(  \tau\right)  \frac{z^{p}(\tau)}{\left(  \frac{1}{2}%
-i\tau\right)  ^{2}}d\tau,\quad p=1,2,\ldots\label{hP0}%
\end{equation}
admit the estimate%
\[
\left\vert h_{p}^{\left(  0\right)  }\right\vert =O\left(  \frac{1}{p^{l}%
}\right)  ,\quad p\rightarrow\infty
\]
with $l=\left\lfloor \frac{4}{5}\alpha+1\right\rfloor $.
\end{theorem}

\begin{proof}
The change of the variable $u=z(\tau)$ in the integral (\ref{hP0}) leads to
the expression for $h_{p}^{\left(  0\right)  }$:
\begin{equation}
h_{p}^{\left(  0\right)  }=\frac{1}{2\pi i}\int_{\mathbb{T}}\widetilde
{\varphi}_{0}^{-}\left(  u\right)  u^{p}du,\label{hP0 new}%
\end{equation}
where
\[
\widetilde{\varphi}_{0}^{-}\left(  u\right)  =\varphi_{0}^{-}\left(
z^{-1}(u)\right)  ,\quad z^{-1}(u)=\frac{1}{2i}\frac{u-1}{u+1}.
\]
Integrating by parts $l$ times (\ref{hP0 new}) we obtain
\begin{equation}
h_{p}^{\left(  0\right)  }=\frac{\left(  -1\right)  ^{l}}{2\pi i}\frac
{1}{\left(  p+1\right)  \left(  p+2\right)  \cdots\left(  p+l\right)  }%
\int_{\mathbb{T}}\left(  \widetilde{\varphi}_{0}^{-}\left(  u\right)  \right)
^{\left(  l\right)  }u^{p+l}du.\label{hD}%
\end{equation}
Let us show that the conditions of the theorem guarantee the inclusion
$\left(  \widetilde{\varphi}_{0}^{-}\left(  u\right)  \right)  ^{\left(
l\right)  }\in L_{1}\left(  \mathbb{T}\right)  $. For this let us return to
the real axis taking into account that
\begin{align*}
\widetilde{\varphi}_{0}^{-}\left(  u\right)   &  =\frac{1}{2\pi i}%
\int_{\mathbb{R}}\frac{\varphi_{0}\left(  \tau\right)  }{\tau-\frac{1}%
{2i}\frac{u-1}{u+1}}d\tau\\
&  =\frac{1}{2\pi i}\int_{\mathbb{T}}\frac{\varphi_{0}\left(  z^{-1}%
(v)\right)  }{\frac{1}{2i}\frac{v-1}{v+1}-\frac{1}{2i}\frac{u-1}{u+1}}%
\frac{dv}{i\left(  v+1\right)  ^{2}}\\
&  =\frac{1}{2\pi i}\int_{\mathbb{T}}\frac{\varphi_{0}\left(  z^{-1}%
(v)\right)  }{v-u}\frac{u+1}{v+1}dv\\
&  =\frac{1}{2\pi i}\int_{\mathbb{T}}\frac{\varphi_{0}\left(  z^{-1}%
(v)\right)  }{v-u}dv-\frac{1}{2\pi i}\int_{\mathbb{T}}\frac{\varphi_{0}\left(
z^{-1}(v)\right)  }{v+1}dv.
\end{align*}
Thus, since the second term here is constant, the derivatives of this function
are calculated by the formula%
\[
\left(  \widetilde{\varphi}_{0}^{-}\left(  u\right)  \right)  ^{\left(
l\right)  }=\frac{\left(  l-1\right)  !}{2\pi i}\int_{\mathbb{T}}\frac
{\varphi_{0}\left(  z^{-1}(v)\right)  }{\left(  v-u\right)  ^{l+1}}dv.
\]
The inverse change of the variable $v=z(\tau)$, $u=z(\rho)$ gives
\begin{align}
\left(  \widetilde{\varphi}_{0}^{-}\left(  z(\rho)\right)  \right)  ^{\left(
l\right)  } &  =-\frac{\left(  l-1\right)  !}{2\pi}\int_{\mathbb{R}}%
\frac{\varphi_{0}\left(  \tau\right)  }{\left(  \frac{\frac{1}{2}+i\tau}%
{\frac{1}{2}-i\tau}-\frac{\frac{1}{2}+i\rho}{\frac{1}{2}-i\rho}\right)
^{l+1}}\frac{d\tau}{\left(  \frac{1}{2}-i\tau\right)  ^{2}}\nonumber\\
&  =-\frac{\left(  l-1\right)  !}{2\pi}\left(  -i\right)  ^{l+1}\left(
\frac{1}{2}-i\rho\right)  ^{l+1}\int_{\mathbb{R}}\frac{\varphi_{0}\left(
\tau\right)  \left(  \frac{1}{2}-i\tau\right)  ^{l-1}}{\left(  \tau
-\rho\right)  ^{l+1}}d\tau\nonumber\\
&  =c_{l}\left(  \frac{1}{2}-i\rho\right)  ^{l+1}J_{l}\left(  \rho\right)
,\label{phi0tilde}%
\end{align}
where $c_{l}=-\frac{\left(  l-1\right)  !}{2\pi}\left(  -i\right)  ^{l+1}$
and
\begin{equation}
J_{l}\left(  \rho\right)     :=\int_{\mathbb{R}}\frac{\varphi_{0}\left(
\tau\right)  \left(  \frac{1}{2}-i\tau\right)  ^{l-1}}{\left(  \tau
-\rho\right)  ^{l+1}}d\tau  =\sum_{j=0}^{l-1}C_{l-1}^{j}\left(  \frac{1}{2}\right)  ^{j}\left(
-i\right)  ^{l-1-j}\int_{\mathbb{R}}\frac{\varphi_{0}\left(  \tau\right)
\tau^{l-1-j}}{\left(  \tau-\rho\right)  ^{l+1}}d\tau.\label{sum of integrals}%
\end{equation}
Consider the integral corresponding to $j=0$,%
\begin{equation}
J_{l,0}\left(  \rho\right)  :=\int_{\mathbb{R}}\frac{\varphi_{0}\left(
\tau\right)  \tau^{l-1}}{\left(  \tau-\rho\right)  ^{l+1}}d\tau.\label{Jl0}%
\end{equation}
Let us study the behaviour of (\ref{Jl0}) when $\rho\rightarrow\infty$ using
the method from \cite{GR2018MathNotes}. Substituting the integral (\ref{G-})
into (\ref{Jl0}) and interchanging the order of integration we obtain%
\begin{equation}
J_{l,0}\left(  \rho\right)  =\int_{0}^{\infty}\left(  q(s)I_{l}^{\left(
1\right)  }(s,\rho)+Q^{\prime}(s)I_{l}^{\left(  2\right)  }(s,\rho)\right)
ds,\label{Jl0 in q}%
\end{equation}
where
\begin{equation}
I_{l}^{\left(  1\right)  }(s,\rho)=\frac{1}{2i}\int_{\mathbb{R}}%
\frac{e^{iS(\tau,x-s)}\tau^{l-2}}{\left(  \tau-\rho\right)  ^{l+1}}%
d\tau\label{Il1}%
\end{equation}
and
\begin{equation}
I_{l}^{\left(  2\right)  }(s,\rho)=\frac{1}{\left(  2i\right)  ^{2}}%
\int_{\mathbb{R}}\frac{e^{iS(\tau,x-s)}\tau^{l-3}}{\left(  \tau-\rho\right)
^{l+1}}d\tau.\label{Il2}%
\end{equation}
Consider (\ref{Il1}), where we change the variable:%
\begin{equation}
\tau=\beta\left(  s\right)  u,\quad\rho=\beta\left(  s\right)  \xi
\quad\text{with }\beta\left(  s\right)  =\left(  \frac{s-x}{12t}\right)
^{\frac{1}{2}}.\label{change variable tau=}%
\end{equation}
Setting
\[
\sigma(u)=\frac{u^{3}}{3}-u,\quad\Lambda\left(  s,x\right)  =\Lambda\left(
s\right)  =\frac{\left(  s-x\right)  ^{\frac{3}{2}}}{\left(  3t\right)
^{\frac{1}{2}}}%
\]
we obtain%
\[
I_{l}^{\left(  1\right)  }(s,\rho)=\frac{\beta^{-2}\left(  s\right)  }{2i}%
\int_{\mathbb{R}}\frac{u^{l-2}e^{i\Lambda\left(  s\right)  \sigma(u)}}{\left(
u-\xi\right)  ^{l+1}}du.
\]
Here $\Lambda\left(  s\right)  $ can be regarded as a large parameter, and
hence it is natural to apply to this integral the saddle-point method. In
\cite{GR2018MathNotes} that was done in the cases $l=0$ and $l=2$. In a
similar way (see (4.17) and (5.5) in \cite{GR2018MathNotes}) we obtain
\begin{equation}
\left\vert \int_{\mathbb{R}}\frac{u^{l-2}e^{i\Lambda\left(  s\right)
\sigma(u)}}{\left(  u-\xi\right)  ^{l+1}}du\right\vert \leq
\operatorname*{Const}\left\{
\begin{array}
[c]{c}%
\frac{1}{\left\vert \xi\mp1\right\vert ^{l+1}\Lambda^{\frac{1}{2}}\left(
s\right)  },\quad\left\vert \xi\mp1\right\vert \Lambda^{\frac{1}{2}}\left(
s\right)  \geq1,\\
\Lambda^{\frac{l}{2}}\left(  s\right)  ,\quad\left\vert \xi\mp1\right\vert
\Lambda^{\frac{1}{2}}\left(  s\right)  \leq1.
\end{array}
\right.  \label{bound integral u}%
\end{equation}
These bounds allow us to estimate (\ref{Jl0 in q}). Let us obtain a bound for
the first integral in (\ref{Jl0 in q}). Denoting it as $J_{l,0}^{\left(
1\right)  }\left(  \rho\right)  $ we have%
\begin{align}
\left\vert J_{l,0}^{\left(  1\right)  }\left(  \rho\right)  \right\vert  &
\leq\operatorname*{Const}(\int_{M^{+}}\left\vert q(s)\right\vert \beta
^{-2}(s)\Lambda^{\frac{l}{2}}\left(  s\right)  ds+\int_{\mathbb{R}\backslash
M^{+}}\frac{\left\vert q(s)\right\vert \beta^{-2}(s)}{\left\vert
\xi-1\right\vert ^{l+1}\Lambda^{\frac{1}{2}}\left(  s\right)  }ds\nonumber\\
&  +\int_{M^{-}}\left\vert q(s)\right\vert \beta^{-2}(s)ds+\int_{\mathbb{R}%
\backslash M^{-}}\frac{\left\vert q(s)\right\vert \beta^{-2}(s)}{\left\vert
\xi+1\right\vert ^{l+1}\Lambda^{\frac{1}{2}}\left(  s\right)  }%
ds),\label{Jl0 1}%
\end{align}
where $M^{\pm}=\left\{  s\in\left(  0,\infty\right)  \mid\left\vert \xi
\mp1\right\vert \Lambda^{\frac{1}{2}}\left(  s\right)  \leq1\right\}  $. Since
$\xi=\rho/\beta(s)$, the sets $M^{\pm}$ can be described as follows%
\begin{equation}
M^{\pm}=\left\{  s\in\left(  0,\infty\right)  \mid\left\vert \rho\mp
\beta(s)\right\vert \leq\Lambda^{-\frac{1}{2}}\left(  s\right)  \beta
(s)\right\}  .\label{Mpm}%
\end{equation}
According to the definition of the functions $\beta(s)$ and $\Lambda\left(
s\right)  $,
\[
\Lambda^{-\frac{1}{2}}\left(  s\right)  \beta(s)=d(t)\beta^{-\frac{1}{2}}(s)
\]
where $d(t)$ is calculated by the formula%
\begin{equation}
d=\left(  \Lambda^{-\frac{1}{2}}(s)\beta(s)\right)  \beta^{\frac{1}{2}%
}(s)=\left(  \frac{\beta^{3}(s)}{\Lambda(s)}\right)  ^{\frac{1}{2}}=\frac
{1}{4\sqrt[4]{t/2}}.\label{d(t)}%
\end{equation}
\qquad Thus, the inequalities defining the sets $M^{\pm}$ have the form%
\begin{equation}
\left\vert \rho\mp\beta(s)\right\vert \leq d(t)\beta^{-\frac{1}{2}%
}(s).\label{ineq rho-beta}%
\end{equation}
Let us write this condition (for the \textquotedblleft$-$\textquotedblright%
\ sign) as follows%
\begin{equation}
F_{-}(s):=\beta(s)-d(t)\beta^{-\frac{1}{2}}(s)\leq\rho\leq\beta(s)+d(t)\beta
^{-\frac{1}{2}}(s)=:F_{+}(s).\label{ineq F 1}%
\end{equation}
Obviously, the functions $F_{\pm}(s)$ are monotonous for large enough $s$.
Hence%
\begin{equation}
F_{+}^{-1}(\rho)\leq s\leq F_{-}^{-1}(\rho),\label{ineq F}%
\end{equation}
and it is easy to see that
\begin{equation}
F_{\pm}^{-1}(\rho)\sim c\rho^{2}\label{asympt F}%
\end{equation}
for some $c>0$. Let us estimate the first integral in (\ref{Jl0 1}). Taking
into account (\ref{ineq F}) we obtain
\begin{align*}
\int_{M^{+}}\left\vert q(s)\right\vert \beta^{-2}(s)\Lambda^{\frac{l}{2}%
}\left(  s\right)  ds &  \leq\operatorname*{Const}\int_{M^{+}}\left\vert
q(s)\right\vert \left(  1+\left\vert s\right\vert \right)  ^{-1+\frac{3l}{4}%
}ds\\
&  \leq\operatorname*{Const}\int_{M^{+}}\left\vert q(s)\right\vert \left(
1+\left\vert s\right\vert \right)  ^{\alpha_{1}+\left(  -1+\frac{3l}{4}%
-\alpha_{1}\right)  }ds\\
&  \leq\operatorname*{Const}\left(  1+\left\vert \rho\right\vert \right)
^{\frac{3l}{2}-2-2\alpha_{1}}\int_{M^{+}}\left\vert q(s)\right\vert \left(
1+\left\vert s\right\vert \right)  ^{\alpha_{1}}ds.
\end{align*}
Note that the last integral exists for all $\alpha_{1}\geq0$ since the
interval $M^{+}$ is finite. This integral is a function of $\rho$. Let us
denote it by $\mathbf{L}_{1}^{+}(\rho)$ and find the maximum value of
$\alpha_{1}$ for which $\mathbf{L}_{1}^{+}(\rho)\in L_{1}\left(
\mathbb{R}\right)  $. For $\rho>0$ we have%
\begin{equation*}
\left\Vert \mathbf{L}_{1}^{+}\right\Vert _{L_{1}}   =\int_{0}^{\infty
}\mathbf{L}_{1}^{+}(\rho)d\rho=\int_{0}^{\infty}\int_{M^{+}}\left\vert
q(s)\right\vert \left(  1+s\right)  ^{\alpha_{1}}dsd\rho
  =\int_{0}^{\infty}\left\vert q(s)\right\vert \left(  1+s\right)
^{\alpha_{1}}\left(  \int_{m^{+}}d\rho\right)  ds
\end{equation*}
where the interval $m^{+}:=(\beta(s)-d(t)\beta^{-\frac{1}{2}}(s),\,\beta
(s)+d(t)\beta^{-\frac{1}{2}}(s))$ is defined by the inequalities
(\ref{ineq F 1}). Hence%
\begin{equation*}
\left\Vert \mathbf{L}_{1}^{+}\right\Vert _{L_{1}\left(  \mathbb{R}\right)  }
 \leq\operatorname*{Const}\int_{0}^{\infty}\left\vert q(s)\right\vert
\left(  1+\left\vert s\right\vert \right)  ^{\alpha_{1}}\beta^{-\frac{1}{2}%
}(s)ds  \leq\operatorname*{Const}\int_{0}^{\infty}\left\vert q(s)\right\vert
\left(  1+\left\vert s\right\vert \right)  ^{\alpha_{1}-\frac{1}{4}}ds.
\end{equation*}
Thus, due to (\ref{cond q}), the last integral converges if $\alpha_{1}%
-\frac{1}{4}\leq\alpha$. Choosing the largest possible value $\alpha
_{1}=\alpha+\frac{1}{4}$ we obtain%
\begin{equation}
\int_{M^{+}}\left\vert q(s)\right\vert \beta^{-2}(s)\Lambda^{\frac{l}{2}%
}\left(  s\right)  ds\leq\operatorname*{Const}\left(  1+\left\vert
\rho\right\vert \right)  ^{\frac{3l}{2}-\frac{5}{2}-2\alpha}\mathbf{L}_{1}%
^{+}(\rho)\label{bound1}%
\end{equation}
where $\mathbf{L}_{1}^{+}(\rho)\in L_{1}\left(  \mathbb{R}\right)  $. Now let
us estimate the second integral in (\ref{Jl0 1}). According to (\ref{ineq F 1}%
) it splits in two integrals over the sets $s\geq F_{-}^{-1}(\rho)$ and $0\leq
s\leq F_{+}^{-1}(\rho)$ where $F_{\pm}^{-1}(\rho)$ satisfy the asymptotic
relations (\ref{asympt F}). Furthermore,
\begin{align*}
\int_{F_{-}^{-1}(\rho)}^{\infty}\frac{\left\vert q(s)\right\vert \beta
^{-2}(s)}{\left\vert \xi-1\right\vert ^{l+1}\Lambda^{\frac{1}{2}}\left(
s\right)  }ds &  =\int_{F_{-}^{-1}(\rho)}^{\infty}\frac{\left\vert
q(s)\right\vert \beta^{l-1}(s)\Lambda^{-\frac{1}{2}}\left(  s\right)
}{\left\vert \rho-\beta(s)\right\vert ^{l+1}}ds\\
&  \leq\operatorname*{Const}\int_{F_{-}^{-1}(\rho)}^{\infty}\frac{\left\vert
q(s)\right\vert \left(  1+s\right)  ^{\frac{l}{2}-\frac{5}{4}}}{\left\vert
\rho-\beta(s)\right\vert ^{l+1}}ds\\
&  \leq\operatorname*{Const}\left(  1+\left\vert \rho\right\vert \right)
^{l-\frac{5}{2}-2\alpha_{2}}\int_{F_{-}^{-1}(\rho)}^{\infty}\frac{\left\vert
q(s)\right\vert \left(  1+s\right)  ^{\alpha_{2}}}{\left\vert \rho
-\beta(s)\right\vert ^{l+1}}ds.
\end{align*}
Let us find the largest possible value of $\alpha_{2}\geq0$ such that the last
integral
\[
\mathbf{L}_{2}^{+}(\rho):=\int_{F_{-}^{-1}(\rho)}^{\infty}\frac{\left\vert
q(s)\right\vert \left(  1+s\right)  ^{\alpha_{2}}}{\left\vert \rho
-\beta(s)\right\vert ^{l+1}}ds
\]
be a function of the class $L_{1}\left(  \mathbb{R}\right)  $. We have%
\begin{align*}
\left\Vert \mathbf{L}_{2}^{+}(\rho)\right\Vert _{L_{1}\left(  \mathbb{R}%
\right)  } &  =\int_{0}^{\infty}\mathbf{L}_{2}^{+}(\rho)d\rho=\int_{0}%
^{\infty}q(s)(1+s)^{\alpha_{2}}\left(  \int_{0}^{F_{-}(s)}\frac{d\rho}{\left(
\beta(s)-\rho\right)  ^{l+1}}\right)  ds\\
&  \leq\operatorname*{Const}\int_{0}^{\infty}q(s)(1+s)^{\alpha_{2}}%
\beta^{\frac{l}{2}}(s)ds\leq\operatorname*{Const}\int_{0}^{\infty
}q(s)(1+s)^{\alpha_{2}+\frac{l}{4}}ds.
\end{align*}
The last integral converges when
\begin{equation}
\alpha_{2}+\frac{l}{4}\leq\alpha.\label{ineq alpha2}%
\end{equation}
Thus,
\begin{equation}
\int_{F_{-}^{-1}(\rho)}^{\infty}\frac{\left\vert q(s)\right\vert \beta
^{-2}(s)}{\left\vert \xi-1\right\vert ^{l+1}\Lambda^{\frac{1}{2}}\left(
s\right)  }ds\leq\operatorname*{Const}\left(  1+\left\vert \rho\right\vert
\right)  ^{l-\frac{5}{2}-2\alpha_{2}}\mathbf{L}_{2}^{+}(\rho)\label{bound L2}%
\end{equation}
where $\mathbf{L}_{2}^{+}(\rho)\in L_{1}\left(  \mathbb{R}\right)  $ and
$\alpha_{2}$ satisfies (\ref{ineq alpha2}). The largest value of $\alpha_{2}$
satisfying (\ref{ineq alpha2}) is $\alpha_{2}=\alpha-\frac{l}{4}$. Thus the
bound (\ref{bound L2}) takes the form%
\begin{equation}
\int_{F_{-}^{-1}(\rho)}^{\infty}\frac{\left\vert q(s)\right\vert \beta
^{-2}(s)}{\left\vert \xi-1\right\vert ^{l+1}\Lambda^{\frac{1}{2}}\left(
s\right)  }ds\leq\operatorname*{Const}\left(  1+\left\vert \rho\right\vert
\right)  ^{\frac{3}{2}l-\frac{5}{2}-2\alpha}\mathbf{L}_{2}^{+}(\rho
)\label{bound L2 main}%
\end{equation}
where $\mathbf{L}_{2}^{+}(\rho)\in L_{1}\left(  \mathbb{R}\right)  $. Consider
the second part of the second integral in (\ref{Jl0 1}),%
\begin{equation}
\int_{0}^{F_{+}^{-1}(\rho)}\frac{\left\vert q(s)\right\vert \beta^{-2}%
(s)}{\left\vert \xi-1\right\vert ^{l+1}\Lambda^{\frac{1}{2}}\left(  s\right)
}ds=\int_{\frac{\rho^{2}}{2}}^{F_{+}^{-1}(\rho)}\frac{\left\vert
q(s)\right\vert \beta^{-2}(s)}{\left(  1-\xi\right)  ^{l+1}\Lambda^{\frac
{1}{2}}\left(  s\right)  }ds+\int_{0}^{\frac{\rho^{2}}{2}}\frac{\left\vert
q(s)\right\vert \beta^{-2}(s)}{\left(  1-\xi\right)  ^{l+1}\Lambda^{\frac
{1}{2}}\left(  s\right)  }ds.\label{int22}%
\end{equation}
The first integral here admits the bound (\ref{bound L2 main}) which is proved
similarly. Let us estimate the second integral%
\begin{align*}
\int_{0}^{\frac{\rho^{2}}{2}}\frac{\left\vert q(s)\right\vert \beta^{-2}%
(s)}{\left(  1-\xi\right)  ^{l+1}\Lambda^{\frac{1}{2}}\left(  s\right)  }ds &
=\int_{0}^{\frac{\rho^{2}}{2}}\frac{\left\vert q(s)\right\vert \beta
^{l-1}(s)\Lambda^{-\frac{1}{2}}\left(  s\right)  }{\left(  \beta
(s)-\rho\right)  ^{l+1}}ds\\
&  \leq\operatorname*{Const}\frac{1}{\left(  1+\left\vert \rho\right\vert
\right)  ^{l+1}}\int_{0}^{\frac{\rho^{2}}{2}}\left\vert q(s)\right\vert
\beta^{l-1}(s)\Lambda^{-\frac{1}{2}}\left(  s\right)  ds\\
&  \leq\operatorname*{Const}\frac{1}{\left(  1+\left\vert \rho\right\vert
\right)  ^{l+1}}\int_{0}^{\frac{\rho^{2}}{2}}\left\vert q(s)\right\vert
\left(  1+\left\vert s\right\vert \right)  ^{\frac{l}{2}-\frac{5}{4}}ds.
\end{align*}
The last integral is uniformly bounded with respect to $\rho$ when
\begin{equation}
\frac{l}{2}-\frac{5}{4}\leq\alpha.\label{ineq l/2}%
\end{equation}
In this case%
\begin{equation}
\int_{0}^{\frac{\rho^{2}}{2}}\frac{\left\vert q(s)\right\vert \beta^{-2}%
(s)}{\left(  1-\xi\right)  ^{l+1}\Lambda^{\frac{1}{2}}\left(  s\right)
}ds\leq\operatorname*{Const}\frac{1}{\left(  1+\left\vert \rho\right\vert
\right)  ^{l+1}}.\label{bound vsp}%
\end{equation}
Taking into account the estimates (\ref{bound1}) and (\ref{bound L2 main})
(valid for the integral on the left hand side in (\ref{bound L2 main}) and for
the first integral in (\ref{int22})), as well as the estimate (\ref{ineq l/2}%
)-(\ref{bound vsp}) we obtain that the first two terms in (\ref{Jl0 1}) admit
the estimate%
\begin{equation}
\begin{split}
\int_{M^{+}}\left\vert q(s)\right\vert \beta^{-2}(s)\Lambda^{\frac{l}{2}%
}\left(  s\right)  ds&+\int_{\mathbb{R}\backslash M^{+}}\frac{\left\vert
q(s)\right\vert \beta^{-2}(s)}{\left\vert \xi-1\right\vert \Lambda^{\frac
{1}{2}}\left(  s\right)  }ds\\
&\leq\operatorname*{Const}\left(  \left(
1+\left\vert \rho\right\vert \right)  ^{\frac{3}{2}l-\frac{5}{2}-2\alpha
}\mathbf{L}(\rho)+\left(  1+\left\vert \rho\right\vert \right)  ^{-\left(
l+1\right)  }\right)  ,
\end{split}
\label{estimate M}%
\end{equation}
where $\mathbf{L}(\rho)\in L_{1}\left(  \mathbb{R}\right)  $ and
(\ref{ineq l/2}) is valid. In a similar way it is shown that the other two
terms (\ref{Jl0 1}) also admit the estimate (\ref{estimate M}). Thus,
\begin{equation}
\left\vert J_{l,0}^{\left(  1\right)  }\left(  \rho\right)  \right\vert
\leq\operatorname*{Const}\left(  \left(  1+\left\vert \rho\right\vert \right)
^{\frac{3}{2}l-\frac{5}{2}-2\alpha}\mathbf{L}(\rho)+\left(  1+\left\vert
\rho\right\vert \right)  ^{-\left(  l+1\right)  }\right)
\label{estimate Jl01}%
\end{equation}
with $\mathbf{L}(\rho)\in L_{1}\left(  \mathbb{R}\right)  $ and
(\ref{ineq l/2}). Let us recall that $J_{l,0}^{\left(  1\right)  }\left(
\rho\right)  $ is a first part of the integral (\ref{Jl0 in q}). Its second
part is estimated analogously by taking into account the following two
observations. 1. Due to the condition (\ref{Qprime}) it is easy to show that
\[
Q^{\prime}(s)\in L_{1}\left(  \mathbb{R},\left(  1+\left\vert \rho\right\vert
\right)  ^{\alpha-1}\right)  .
\]
2. The integrals (\ref{Il1}) and (\ref{Il2}) differ by the factors $\tau
^{l-2}$ and $\tau^{l-3}$, respectively. Hence the estimate for the integral
\[
J_{l,0}^{\left(  2\right)  }\left(  \rho\right)  =\int_{0}^{\infty}Q^{\prime
}(s)I_{l}^{\left(  2\right)  }(s,\rho)ds
\]
is obtained from (\ref{estimate Jl01}) by replacing $\alpha$ with $(\alpha-1)$
and decreasing of the power of $\beta(s)$ by one. Thus we have
\begin{align}
\left\vert J_{l,0}^{\left(  2\right)  }\left(  \rho\right)  \right\vert  &
\leq\operatorname*{Const}\left(  \left(  1+\left\vert \rho\right\vert \right)
^{\frac{3}{2}l-\frac{5}{2}-1-2\left(  \alpha-1\right)  }\mathbf{L}%
(\rho)+\left(  1+\left\vert \rho\right\vert \right)  ^{-\left(  l+1\right)
}\right)  \nonumber\\
&  =\operatorname*{Const}\left(  \left(  1+\left\vert \rho\right\vert \right)
^{\frac{3}{2}l-\frac{3}{2}-2\alpha}\mathbf{L}(\rho)+\left(  1+\left\vert
\rho\right\vert \right)  ^{-\left(  l+1\right)  }\right)
,\label{estimate Jl02}%
\end{align}
where the condition (\ref{ineq l/2}) is replaced by
\begin{equation}
\frac{l}{2}-\frac{7}{4}\leq\alpha.\label{ineq l/2 2}%
\end{equation}
Comparing and summing up the estimates (\ref{estimate Jl01}) and
(\ref{estimate Jl02}) we obtain (see (\ref{Jl0 in q})) that%
\begin{equation}
\left\vert J_{l,0}\left(  \rho\right)  \right\vert \leq\operatorname*{Const}%
\left(  \left(  1+\left\vert \rho\right\vert \right)  ^{\frac{3}{2}l-\frac
{3}{2}-2\alpha}\mathbf{L}(\rho)+\left(  1+\left\vert \rho\right\vert \right)
^{-\left(  l+1\right)  }\right)  ,\label{estimate Jl0}%
\end{equation}
where $\mathbf{L}(\rho)\in L_{1}\left(  \mathbb{R}\right)  $ and
(\ref{ineq l/2}) is fulfilled ((\ref{ineq l/2}) is stronger than
(\ref{ineq l/2 2})). Let us recall that the integral $J_{l,0}\left(
\rho\right)  $ of the form (\ref{Jl0}) is one from the group of integrals
$J_{l,j}\left(  \rho\right)  $, $j=0,1,\ldots,l-1$ from the sum
(\ref{sum of integrals}) of integrals having the form
\[
J_{l,j}\left(  \rho\right)  =\int_{\mathbb{R}}\frac{\varphi_{0}\left(
\tau\right)  \tau^{l-1-j}}{\left(  \tau-\rho\right)  ^{l+1}}d\tau.
\]
Obviously these integrals for $j=1,2,\ldots,l-1$ also admit the estimate
(\ref{estimate Jl0}), (\ref{ineq l/2}). Hence the integral $J_{l}\left(
\rho\right)  $ from (\ref{phi0tilde}) admits this estimate as well. Thus,%
\begin{equation}
\left\vert \widetilde{\varphi}_{0}^{-}\left(  z(\rho)\right)  \right\vert
=c_{l}\left\vert \frac{1}{2}-i\rho\right\vert ^{l+1}\left\vert J_{l}\left(
\rho\right)  \right\vert  \leq\operatorname*{Const}\left(  \left(  1+\left\vert \rho\right\vert
\right)  ^{\frac{5}{2}l-\frac{1}{2}-2\alpha}\mathbf{L}(\rho)+1\right)
\label{estimate phi0tilde}%
\end{equation}
where $\mathbf{L}(\rho)\in L_{1}\left(  \mathbb{R}\right)  $ and
(\ref{ineq l/2}) is fulfilled. Let us return to formula (\ref{hD}),%
\begin{equation*}
\left\vert h_{p}^{(0)}\right\vert    \leq\frac{\operatorname*{Const}}{p^{l}%
}\int_{\mathbb{T}}\left\vert \widetilde{\varphi}_{0}^{-(l)}\left(
z(u)\right)  \right\vert \left\vert du\right\vert  =\frac{\operatorname*{Const}}{p^{l}}\int_{\mathbb{R}}\left\vert
\widetilde{\varphi}_{0}^{-(l)}\left(  \rho\right)  \right\vert \frac{d\rho
}{\left\vert \frac{1}{2}-i\rho\right\vert ^{2}}.
\end{equation*}
Due to the estimate (\ref{estimate phi0tilde}) we have
\[
\left\vert h_{p}^{(0)}\right\vert \leq\frac{\operatorname*{Const}}{p^{l}}%
\int_{\mathbb{R}}\left(  \left(  1+\left\vert \rho\right\vert \right)
^{\frac{5}{2}l-\frac{5}{2}-2\alpha}\mathbf{L}(\rho)+\left(  1+\left\vert
\rho\right\vert \right)  ^{-2}\right)  d\rho.
\]
Obviously the last integral converges when
\begin{equation}
\frac{5}{2}l-\frac{5}{2}\leq2\alpha,\label{ineq 5/2}%
\end{equation}
and in this case
\[
\left\vert h_{p}^{(0)}\right\vert \leq\frac{\operatorname*{Const}}{p^{l}}.
\]
Note that for $l\geq1$ (\ref{ineq 5/2}) implies (\ref{ineq l/2}). Thus Theorem
\ref{Th estimate for hp} is proved.
\end{proof}

We estimated Fourier coefficients of the function $\varphi_{0}^{-}\left(
\rho\right)  $ which has the form (\ref{phi0-}) and hence the entries of the
Hankel matrix $\mathbf{H}(\varphi_{0})$ ($=\mathbf{H}(\varphi_{0}^{-})$). At
the same time in the system (\ref{Equation Hankel 3}) the symbol
\begin{equation}
\varphi\left(  \rho\right)  =T_{+}\left(  \rho\right)  \varphi_{0}\left(
\rho\right)  \label{phi via phi0}%
\end{equation}
appears, where $T_{+}\left(  \rho\right)  \in H_{\infty}^{+}\left(
\mathbb{R}\right)  $. For estimating the entries of the matrix $\mathbf{H}%
(\varphi)$ we use the representation%
\begin{equation}
\mathbf{H}(\varphi)=\mathbf{H}(\varphi_{0})T(T_{+}), \label{Hphi factorized}%
\end{equation}
where $T(T_{+}):=P^{+}T_{+}P^{+}$ is a Toeplitz operator, and $P^{+}$ is an
analytic projection operator $P^{+}:L_{2}\left(  \mathbb{R}\right)
\rightarrow H_{2}^{+}\left(  \mathbb{R}\right)  $ which similarly to $P^{-}$
(see Section \ref{Sect Convergence of method of reduction}) can be written in
the form%
\[
P^{+}f\left(  \rho\right)  =\lim_{\varepsilon\rightarrow0+}\frac{1}{2\pi
i}\int_{-\infty}^{\infty}\frac{f\left(  \tau\right)  }{\tau-\left(
\rho+i\varepsilon\right)  }d\tau.
\]
The matrix of the operator $T(T_{+}):L_{2}\left(  \mathbb{R}\right)
\rightarrow H_{2}^{+}\left(  \mathbb{R}\right)  $ in the basis $\left\{
\frac{z^{n}\left(  \rho\right)  }{\frac{1}{2}-i\rho}\right\}  _{n=0}^{\infty}$
(see (\ref{repr y})) can be written as follows%
\begin{equation}
T(T_{+})=\left(
\begin{array}
[c]{cccccc}%
t_{0} & 0 & 0 & \ldots & 0 & \ldots\\
t_{1} & t_{0} & 0 & \ldots & 0 & \ldots\\
t_{2} & t_{1} & t_{0} & \ldots & 0 & \ldots\\
& \ldots &  & \ddots &  & \\
t_{n} & t_{n-1} & t_{n-2} & \ldots & t_{0} & \ldots\\
& \ldots &  & \ldots &  & \ddots
\end{array}
\right)  \label{matrix T}%
\end{equation}
where%
\[
t_{k}=\frac{1}{2\pi}\int_{-\infty}^{\infty}\frac{T_{+}\left(  \rho\right)
z^{-k}\left(  \rho\right)  }{\left(  \frac{1}{2}-i\rho\right)  ^{2}}%
d\rho,\quad k=0,1,\ldots.
\]
With the aid of the representation (\ref{matrix T}) let us prove the following result.

\begin{theorem}
\label{Th asymptotics hp}Let the reflection coefficient have the form
(\ref{R factorized})-(\ref{G-}), and the condition (\ref{cond q}) be
fulfilled. Then the entries of the Hankel matrix $\mathbf{H}(\varphi)$, where
$\varphi\left(  \rho\right)  $ has the form (\ref{symbol phi}), admit the
estimate
\[
\left\vert h_{p}\right\vert =O\left(  \frac{1}{\left\vert p\right\vert
^{l-\frac{1}{2}}}\right)
\]
where $l=\left\lfloor \frac{4}{5}\alpha+1\right\rfloor $.
\end{theorem}

\begin{Proof}
According to (\ref{Hphi factorized}) the matrix $\mathbf{H}(\varphi)$ is a
product of two matrices%
\[
\mathbf{H}(\varphi_{0}^{-})=\left(
\begin{array}
[c]{cccccc}%
h_{0}^{\left(  0\right)  } & h_{1}^{\left(  0\right)  } & h_{2}^{\left(
0\right)  } & \ldots & h_{n}^{\left(  0\right)  } & \ldots\\
h_{1}^{\left(  0\right)  } & h_{2}^{\left(  0\right)  } & h_{3}^{\left(
0\right)  } & \ldots & h_{n+1}^{\left(  0\right)  } & \ldots\\
h_{2}^{\left(  0\right)  } & h_{3}^{\left(  0\right)  } & h_{4}^{\left(
0\right)  } & \ldots & h_{n+2}^{\left(  0\right)  } & \ldots\\
& \ldots &  & \ddots &  & \\
h_{n}^{\left(  0\right)  } & h_{n+1}^{\left(  0\right)  } & h_{n+2}^{\left(
0\right)  } & \ldots & h_{2n}^{\left(  0\right)  } & \ldots\\
& \ldots &  & \ldots &  & \ddots
\end{array}
\right)
\]
and $T(T_{+})$ of the form (\ref{matrix T}). Thus, an element of the matrix
$\mathbf{H}(\varphi)$ has the form%
\[
h_{p+k}=\sum_{j=k}^{\infty}h_{p+j}t_{j-k}.
\]
Since $T_{+}\left(  \tau\right)  /\left(  \frac{1}{2}-i\tau\right)  \in
L_{2}\left(  \mathbb{R}\right)  $ and
\[
t_{j}=\frac{1}{2\pi}\int_{\mathbb{R}}\frac{T_{+}\left(  \tau\right)
z_{n}^{-k}\left(  \tau\right)  }{\left(  \frac{1}{2}-i\tau\right)  ^{2}}d\tau
\]
we have that $\left\{  t_{j}\right\}  _{j=0}^{\infty}\in l_{2}\left(
\mathbb{Z}\right)  $. Then due to Theorem \ref{Th estimate for hp},%
\begin{equation*}
\left\vert h_{p+k}\right\vert    \leq\left(  \sum_{j=k}^{\infty}\left\vert
h_{p+j}\right\vert ^{2}\right)  ^{\frac{1}{2}}\left\Vert \left\{
t_{j}\right\}  \right\Vert _{l_{2}\left(  \mathbb{Z}\right)  }  \leq\left(  \sum_{j=k}^{\infty}\frac{1}{\left(  p+j\right)  ^{2l}}\right)
^{\frac{1}{2}}\left\Vert \left\{  t_{j}\right\}  \right\Vert _{l_{2}\left(
\mathbb{Z}\right)  }  \leq\frac{\operatorname*{Const}}{\left(  p+k\right)  ^{l-\frac{1}{2}}}.
\end{equation*}

\end{Proof}

Let us introduce the Wiener classes,%
\[
f\left(  \tau\right)  =\sum_{j=-\infty}^{\infty}f_{j}\frac{z^{j}\left(
\tau\right)  }{\frac{1}{2}-i\tau}\in W^{\beta}\left(  \mathbb{R}\right)
\quad\Leftrightarrow\quad\sum_{j=-\infty}^{\infty}\left\vert f_{j}\right\vert
\left(  1+\left\vert j\right\vert \right)  ^{\beta}<\infty,\,\beta\geq0.
\]
Note that the class $W^{\beta}$ is an algebra.

\begin{theorem}
\label{Th applicability of reduction in Wiener classes}Let the reflection
coefficient have the form (\ref{R factorized})--(\ref{G-}) and the condition
(\ref{cond q}) be fulfilled. Then the finite section method is applicable to
system (\ref{Equation Hankel 3}) in the class $W^{\beta}\left(  \mathbb{R}%
\right)  $ for any $0\leq\beta<l-3/2$ with $l=\left\lfloor \frac{4}{5}%
\alpha+1\right\rfloor $.
\end{theorem}

\begin{Proof}
Obviously,
\[
\mathbf{H}(\varphi)=\mathbf{H}(\varphi_{0}^{-}T_{+})=\mathbf{H}(P^{-}\left(
\varphi_{0}^{-}T_{+}\right)  ).
\]
The symbol of the Hankel operator on the right-hand side has the form%
\[
\psi^{-}\left(  \tau\right)  =P^{-}\left(  \varphi_{0}^{-}\left(  \tau\right)
T_{+}\left(  \tau\right)  \right)  =\sum_{p=0}^{\infty}\frac{h_{p}%
z^{-p}\left(  \tau\right)  }{\frac{1}{2}-i\tau}.
\]
Due to Theorem \ref{Th asymptotics hp} this symbol belongs to $W^{\beta
}\left(  \mathbb{R}\right)  $ for $0\leq\beta<l-3/2$. Since $W^{\beta}\left(
\mathbb{R}\right)  $ is an algebra, the operator $\mathbf{H}(\psi^{-})$ is
bounded in the space $W^{\beta}\left(  \mathbb{R}\right)  $. Moreover,
similarly to Theorem \ref{Th H is compact in l1} it can be shown that
$\mathbf{H}(\psi^{-})$ is compact in that space, and the operator
$I+\mathbf{H}(\psi^{-})$ ($=I+\mathbf{H}(\varphi)$) is invertible in
$W^{\beta}\left(  \mathbb{R}\right)  $ ($\beta<l-3/2$). After that, similarly
to Theorem \ref{Th reduction in L2}, the applicability of the finite section
method to (\ref{Equation Hankel 3}) is proved for $W^{\beta}\left(
\mathbb{R}\right)  $.
\end{Proof}

\section{Existence of derivatives $a_{0}^{\prime}$ and $a_{0}^{\prime\prime}$}

When conditions (\ref{R factorized})--(\ref{Qprime}) are fulfilled equation
(\ref{Equation Hankel}) is uniquely solvable in $L_{2}^{+}\left(
\mathbb{R}\right)  $, the symbol $\varphi\left(  \rho\right)  $ having the
form (\ref{symbol phi}). The solution can be written in the form%
\begin{equation}
y\left(  \rho\right)  :=y\left(  \rho,x\right)  =\frac{1}{\frac{1}{2}-i\rho
}\sum_{j=0}^{\infty}a_{j}(x)z^{n}\left(  \rho\right)  ,\quad\sum_{j=0}%
^{\infty}\left\vert a_{j}(x)\right\vert ^{2}<\infty. \label{y}%
\end{equation}
Moreover, due to Theorem \ref{Th H is compact in l1} we have that
\[
\sum_{j=0}^{\infty}\left\vert a_{j}(x)\right\vert <\infty.
\]

\begin{lemma}
\label{Lemma derivatives in LL}Let $\varphi\left(  \rho\right)  $ have the
form (\ref{symbol phi}) and conditions (\ref{R factorized})--(\ref{Qprime}) be
fulfilled. Then
\[
\frac{\partial^{k}\varphi_{0}^{-}\left(  \rho,x\right)  }{\partial x^{k}}\in
L_{2}\left(  \mathbb{R}\right)  \cap L_{\infty}\left(  \mathbb{R}\right)
\]
where $k=1,2,\ldots,\left\lfloor 2\alpha\right\rfloor $, and the function
$\varphi_{0}^{-}\left(  \rho,x\right)  $ has the form (\ref{phi0-}).
\end{lemma}

\begin{proof}
According to (\ref{phi0}), (\ref{phi0-}) we have%
\[
\frac{\partial^{k}\varphi_{0}^{-}\left(  \rho,x\right)  }{\partial x^{k}%
}=\frac{\left(  2i\right)  ^{k}}{2\pi i}\int_{-\infty}^{\infty}\frac{\tau
^{k}G_{-}(\tau)e^{iS(\tau,x)}}{\tau-\rho}d\tau.
\]
Substituting (\ref{G-}) and changing the order of integration we obtain%
\begin{equation}
\frac{\partial^{k}\varphi_{0}^{-}\left(  \rho,x\right)  }{\partial x^{k}%
}=\frac{\left(  2i\right)  ^{k}}{2\pi i}\int_{0}^{\infty}\left(
q(s)L_{k}^{\left(  1\right)  }\left(  s,\rho\right)  +Q^{\prime}%
(s)L_{k}^{\left(  2\right)  }\left(  s,\rho\right)  \right)  ds\label{der=int}%
\end{equation}
where%
\begin{equation}
L_{k}^{\left(  1\right)  }\left(  s,\rho\right)  =\int_{-\infty}^{\infty}%
\frac{e^{iS(\tau,x-s)}\tau^{k-1}}{\tau-\rho}d\tau\label{Lk1}%
\end{equation}
and
\begin{equation}
L_{k}^{\left(  2\right)  }\left(  s,\rho\right)  =\int_{-\infty}^{\infty}%
\frac{e^{iS(\tau,x-s)}\tau^{k-2}}{\tau-\rho}d\tau.\label{Lk2}%
\end{equation}
Let us apply the change of the variable (\ref{change variable tau=}) to the
integral $L_{k}^{\left(  2\right)  }\left(  s,\rho\right)  $. We obtain%
\[
L_{k}^{\left(  2\right)  }\left(  s,\rho\right)  =\beta^{k-2}(s)\int_{-\infty
}^{\infty}\frac{u^{k-2}e^{i\Lambda(s)\sigma(u)}}{u-\xi}du.
\]
Similarly to (\ref{bound integral u}) we have%
\[
L_{k}^{\left(  2\right)  }\left(  s,\rho\right)  \leq\operatorname*{Const}%
\beta^{k-2}(s)\left\{
\begin{array}
[c]{c}%
\frac{1}{\left\vert \xi\mp1\right\vert \Lambda^{\frac{1}{2}}\left(  s\right)
},\quad\left\vert \xi\mp1\right\vert \Lambda^{\frac{1}{2}}\left(  s\right)
\geq1,\\
1,\quad\left\vert \xi\mp1\right\vert \Lambda^{\frac{1}{2}}\left(  s\right)
\leq1.
\end{array}
\right.
\]
Thus, the second part of the integral (\ref{der=int}) can be estimated as
follows%
\begin{align}
\left\vert \int_{0}^{\infty}Q^{\prime}(s)L_{k}^{\left(  2\right)  }\left(
s,\rho\right)  ds\right\vert  &  \leq\operatorname*{Const}(\int_{M^{+}%
}\left\vert Q^{\prime}(s)\right\vert \beta^{k-2}(s)ds+\int_{\mathbb{R}%
\backslash M^{+}}\frac{\left\vert Q^{\prime}(s)\right\vert \beta^{k-2}%
(s)}{\left\vert \xi-1\right\vert \Lambda^{\frac{1}{2}}\left(  s\right)
}ds\nonumber\\
&  +\int_{M^{-}}\left\vert Q^{\prime}(s)\right\vert \beta^{k-2}(s)ds+\int
_{\mathbb{R}\backslash M^{-}}\frac{\left\vert Q^{\prime}(s)\right\vert
\beta^{k-2}(s)}{\left\vert \xi+1\right\vert \Lambda^{\frac{1}{2}}\left(
s\right)  }ds\label{int2}%
\end{align}
where the sets $M^{\pm}$ are defined by (\ref{Mpm}). Denote the first integral
in (\ref{int2}) by $\mathcal{L}_{1}^{+}\left(  \rho\right)  $. Then
\[
\mathcal{L}_{1}^{+}\left(  \rho\right)  \leq\operatorname*{Const}\int_{M^{+}%
}\left\vert Q^{\prime}(s)\right\vert \left(  1+\left\vert s\right\vert
\right)  ^{\frac{k-2}{2}}(s)ds.
\]
Thus, if $\frac{k-2}{2}\leq\alpha-1$ or which is the same $k\leq2\alpha$, then
the function $\mathcal{L}_{1}^{+}\left(  \rho\right)  \in L_{\infty}\left(
\mathbb{R}\right)  $. Let us find now the condition of belonging of
$\mathcal{L}_{1}^{+}\left(  \rho\right)  $ to $L_{2}\left(  \mathbb{R}\right)
$. We have
\begin{equation}
\left\Vert \mathcal{L}_{1}^{+}\left(  \rho\right)  \right\Vert _{L_{2}\left(
\mathbb{R}\right)  }^{2}=\int_{-\infty}^{\infty}\left(  \int_{M^{+}}\left\vert
Q^{\prime}(s)\right\vert \beta^{k-2}(s)ds\right)  \left(  \int_{M^{+}%
}\left\vert Q^{\prime}(s^{\prime})\right\vert \overline{\beta}^{k-2}%
(s^{\prime})ds^{\prime}\right)  d\rho.\label{norm squared 1}%
\end{equation}
This triple (iterated) integral is taken over the domain defined by the
relations%
\[
\left\{
\begin{array}
[c]{c}%
-\infty<\rho<\infty,\\
a_{l}\left(  \rho\right)  \leq s\leq b_{l}\left(  \rho\right)  ,\\
a_{l}\left(  \rho\right)  \leq s^{\prime}\leq b_{l}\left(  \rho\right)  ,
\end{array}
\right.
\]
where $a_{l}\left(  \rho\right)  :=F_{+}^{-1}(\rho)$ and $b_{l}\left(
\rho\right)  :=F_{-}^{-1}(\rho)$ are defined by (\ref{Mpm}). Indeed,
(\ref{Mpm}) gives us the inequalities%
\[
\beta(s)-d\beta^{-\frac{1}{2}}(s)\leq\rho\leq\beta(s)+d\beta^{-\frac{1}{2}}(s)
\]
where the constant $d$ is calculated by (\ref{d(t)}).
Let us change the order of integration in (\ref{norm squared 1}),%
\[
\left\Vert \mathcal{L}_{1}^{+}\left(  \rho\right)  \right\Vert _{L_{2}\left(
\mathbb{R}\right)  }^{2}=2\int_{0}^{\infty}\int_{0}^{\infty}\left\vert
Q^{\prime}(s)\right\vert \beta^{k-2}(s)\left\vert Q^{\prime}(s^{\prime
})\right\vert \beta^{k-2}(s^{\prime})\left(  \int_{\widetilde{F}_{-}%
}^{\widetilde{F}_{+}}d\rho\right)  dsds^{\prime}%
\]
where $\widetilde{F}_{+}=\min\left(  F_{+}(s),F_{+}(s^{\prime})\right)  $ and
$\widetilde{F}_{-}=\max\left(  F_{-}(s),F_{-}(s^{\prime})\right)  $. Note that
if $s>s^{\prime}$ then $\widetilde{F}_{+}=F_{+}(s^{\prime})$ and
$\widetilde{F}_{-}=F_{-}(s^{\prime})$. If $s\leq s^{\prime}$, then
$\widetilde{F}_{\pm}=F_{\pm}(s)$. Thus we obtain%
\begin{align*}
\left\Vert \mathcal{L}_{1}^{+}\left(  \rho\right)  \right\Vert _{L_{2}\left(
\mathbb{R}\right)  }^{2} &  =2\int_{0}^{\infty}\int_{0}^{\infty}\left\vert
Q^{\prime}(s)\right\vert \beta^{k-2}(s)\left\vert Q^{\prime}(s^{\prime
})\right\vert \beta^{k-2}(s^{\prime})\left(  \widetilde{F}_{+}-\widetilde
{F}_{-}\right)  dsds^{\prime}\\
&  =2\int_{0}^{\infty}\left\vert Q^{\prime}(s)\right\vert \beta^{k-2}%
(s)\left(  \int_{0}^{s}\left(  \left\vert Q^{\prime}(s^{\prime})\right\vert
\beta^{k-2}(s^{\prime})\right)  \left(  F_{+}(s^{\prime})-F_{-}(s^{\prime
})ds^{\prime}\right)  \right)  ds\\
&  +2\int_{0}^{\infty}\left\vert Q^{\prime}(s^{\prime})\right\vert \beta
^{k-2}(s^{\prime})\left(  \int_{0}^{s^{\prime}}\left(  \left\vert Q^{\prime
}(s)\right\vert \beta^{k-2}(s)\right)  \left(  F_{+}(s)-F_{-}(s)ds\right)
\right)  ds^{\prime}%
\end{align*}
Since $F_{+}(s)-F_{-}(s)=2d\beta^{-\frac{1}{2}}(s)$ and
\[
\int_{0}^{s}\left\vert Q^{\prime}(s^{\prime})\right\vert \beta^{k-2}%
(s^{\prime})\beta^{-\frac{1}{2}}(s^{\prime})ds^{\prime}\leq\int_{0}^{\infty
}\left\vert Q^{\prime}(s^{\prime})\right\vert \beta^{k-5/2}(s^{\prime
})ds^{\prime},
\]
we obtain%
\begin{align*}
\left\Vert \mathcal{L}_{1}^{+}\left(  \rho\right)  \right\Vert _{L_{2}\left(
\mathbb{R}\right)  }^{2} &  \leq2d(\int_{0}^{\infty}\left\vert Q^{\prime
}(s)\right\vert \beta^{k-2}(s)ds\int_{0}^{\infty}\left\vert Q^{\prime
}(s^{\prime})\right\vert \beta^{k-5/2}(s^{\prime})ds^{\prime}\\
&  +\int_{0}^{\infty}\left\vert Q^{\prime}(s)\right\vert \beta^{k-5/2}%
(s)ds\int_{0}^{\infty}\left\vert Q^{\prime}(s^{\prime})\right\vert \beta
^{k-2}(s^{\prime})ds^{\prime}\\
&  \leq4d\left(  \int_{0}^{\infty}\left\vert Q^{\prime}(s)\right\vert
\beta^{k-5/2}(s)ds\right)  ^{2}.
\end{align*}
Thus,%
\[
\left\Vert \mathcal{L}_{1}^{+}\left(  \rho\right)  \right\Vert _{L_{2}\left(
\mathbb{R}\right)  }\leq\operatorname*{Const}\int_{0}^{\infty}\left\vert
Q^{\prime}(s)\right\vert \left(  1+s\right)  ^{k/2-5/4}ds.
\]
The last integral converges if $k/2-5/4\leq\alpha-1$ or which is the same
$k\leq2\alpha+1/2$. Now let us consider the second integral in (\ref{int2}).
Note that it in turn splits in two integrals, since
\[
\mathbb{R}\backslash M_{l}^{+}=\left(  0,a_{l}\right)  \cup\left(
b_{l},\infty\right)  .
\]
For the integral on $\left(  0,a_{l}\right)  $ we have%
\begin{equation*}
L_{k}^{\left(  3\right)  }\left(  \rho\right)     =\int_{0}^{a_{l}}%
\frac{\left\vert Q^{\prime}(s)\right\vert \beta^{k-2}(s)ds}{\left\vert
\xi-1\right\vert \Lambda^{\frac{1}{2}}(s)}\leq\int_{0}^{a_{l}}\left\vert
Q^{\prime}(s)\right\vert \beta^{k-2}(s)ds  \leq\operatorname*{Const}\int_{0}^{\infty}\left\vert Q^{\prime
}(s)\right\vert \left(  1+s\right)  ^{\frac{k}{2}-1}ds.
\end{equation*}
Thus,
\[
L_{k}^{\left(  3\right)  }\left(  \rho\right)  \in L_{\infty}\left(
\mathbb{R}\right)  \quad\text{if }k\leq2\alpha.
\]
Let us show that additionally under the same condition ($k\leq2\alpha$) the
inclusion $L_{k}^{\left(  3\right)  }\left(  \rho\right)  \in L_{2}\left(
\mathbb{R}\right)  $ is valid as well. We have%
\begin{align*}
&\left\Vert L_{k}^{\left(  3\right)  }\left(  \rho\right)  \right\Vert
_{L_{2}\left(  \mathbb{R}\right)  }^{2}   =\int_{0}^{\infty}\left(  \int
_{0}^{a_{l}}\frac{\left\vert Q^{\prime}(s)\right\vert \beta^{k-2}%
(s)ds}{\left(  \rho-\beta(s)\right)  \Lambda^{\frac{1}{2}}(s)/\beta
(s)}\right)  \left(  \int_{0}^{a_{l}}\frac{\left\vert Q^{\prime}(s^{\prime
})\right\vert \beta^{k-2}(s^{\prime})ds^{\prime}}{\left(  \rho-\beta
(s^{\prime})\right)  \Lambda^{\frac{1}{2}}(s^{\prime})/\beta(s^{\prime}%
)}\right)  d\rho\\
&  =\operatorname*{Const}\int_{0}^{\infty}\int_{0}^{\infty}\left\vert
Q^{\prime}(s)\right\vert \beta^{k-5/2}(s)\left\vert Q^{\prime}(s^{\prime
})\right\vert \beta^{k-5/2}(s^{\prime})\left(  \int_{\mathcal{F}(s,s^{\prime
})}^{\infty}\frac{d\rho}{\left(  \rho-\beta(s)\right)  \left(  \rho
-\beta(s^{\prime})\right)  }\right)  dsds^{\prime}%
\end{align*}
where $\mathcal{F}(s,s^{\prime})=\max(F_{+}(s),F_{+}(s^{\prime}))$. Let us
calculate the interior integral. When $s\geq s^{\prime}$ we have that
$\mathcal{F}(s,s^{\prime})=F_{+}(s)$. Denote
\begin{align*}
\mathcal{B}_{+}\left(  s,s^{\prime}\right)   &  :=\int_{F_{+}(s)}^{\infty
}\frac{d\rho}{\left(  \rho-\beta(s)\right)  \left(  \rho-\beta(s^{\prime
})\right)  }=\frac{1}{\beta(s)-\beta(s^{\prime})}\left.  \ln\frac{\rho
-\beta(s)}{\rho-\beta(s^{\prime})}\right\vert _{\rho=F_{+}(s)}\\
&  =-\frac{1}{\beta(s)-\beta(s^{\prime})}\ln\frac{F_{+}(s)-\beta(s)}%
{F_{+}(s)-\beta(s^{\prime})}=-\frac{1}{\beta(s)-\beta(s^{\prime})}\ln\left(
\frac{d\beta^{-\frac{1}{2}}(s)}{\beta(s)-\beta(s^{\prime})+d\beta^{-\frac
{1}{2}}(s)}\right)  \\
&  =\frac{1}{\beta(s)-\beta(s^{\prime})}\ln\left(  1+\frac{\left(
\beta(s)-\beta(s^{\prime})\right)  \beta^{\frac{1}{2}}(s)}{d}\right)
=\frac{\beta^{\frac{1}{2}}(s)}{d}\frac{1}{z}\ln\left(  1+z\right)
\end{align*}
where $z=\beta^{\frac{1}{2}}(s)\left(  \beta(s)-\beta(s^{\prime})\right)
/d\geq0$. Since
\[
\frac{1}{z}\ln\left(  1+z\right)  <1\quad\text{for }z>0,
\]
we obtain
\begin{equation}
\mathcal{B}_{+}\left(  s,s^{\prime}\right)  \leq\operatorname*{Const}\left\{
\begin{array}
[c]{c}%
\beta^{\frac{1}{2}}(s),\quad s\geq s^{\prime},\\
\beta^{\frac{1}{2}}(s^{\prime}),\quad s\leq s^{\prime}.
\end{array}
\right.  \label{B+ ineq}%
\end{equation}
The second line in (\ref{B+ ineq}) is obtained similarly to the first one.
Thus,%
\begin{align*}
\left\Vert L_{k}^{\left(  3\right)  }\left(  \rho\right)  \right\Vert
_{L_{2}\left(  \mathbb{R}\right)  }^{2} &  \leq\operatorname*{Const}(\int
_{0}^{\infty}\left\vert Q^{\prime}(s)\right\vert \left(  1+s\right)
^{\frac{k}{2}-\frac{5}{4}}\beta^{\frac{1}{2}}(s)\left(  \int_{0}^{s}\left\vert
Q^{\prime}(s^{\prime})\right\vert \left(  1+s^{\prime}\right)  ^{\frac{k}%
{2}-1}ds^{\prime}\right)  ds\\
&  +\int_{0}^{\infty}\left\vert Q^{\prime}(s^{\prime})\right\vert \left(
1+s^{\prime}\right)  ^{\frac{k}{2}-\frac{5}{4}}\beta^{\frac{1}{2}}(s^{\prime
})\left(  \int_{0}^{s^{\prime}}\left\vert Q^{\prime}(s)\right\vert \left(
1+s\right)  ^{\frac{k}{2}-1}ds\right)  ds^{\prime})\\
&  \leq\operatorname*{Const}\left(  \int_{0}^{\infty}\left\vert Q^{\prime
}(s)\right\vert \left(  1+s\right)  ^{\frac{k}{2}-1}ds\right)  ^{2}.
\end{align*}
Hence, if $\frac{k}{2}-1\leq\alpha-1$ $\Leftrightarrow$ $k\leq2\alpha$ we
have
\[
\left\Vert L_{k}^{\left(  3\right)  }\left(  \rho\right)  \right\Vert
_{L_{2}\left(  \mathbb{R}\right)  }^{2}\leq\operatorname*{Const}.
\]
It remains to consider the part of the second integral from (\ref{int2}) on
the interval $\left(  b_{l},\infty\right)  $. Denote%
\[
L_{k}^{\left(  4\right)  }\left(  \rho\right)  :=\int_{b_{l}}^{\infty}%
\frac{\left\vert Q^{\prime}(s)\right\vert \left\vert \beta(s)\right\vert
^{k-2}ds}{\left\vert \xi-1\right\vert \Lambda^{\frac{1}{2}}(s)}.
\]
Analogously to the previous, $L_{k}^{\left(  4\right)  }\left(  \rho\right)
\in L_{\infty}\left(  \mathbb{R}\right)  $ if $k\leq2\alpha$. Let us consider
the question of belonging of $L_{k}^{\left(  4\right)  }\left(  \rho\right)  $
to $L_{2}\left(  \mathbb{R}\right)  $. We have
\[
\left\Vert L_{k}^{\left(  4\right)  }\left(  \rho\right)  \right\Vert
_{L_{2}\left(  \mathbb{R}\right)  }^{2}=\int_{0}^{\infty}\int_{0}^{\infty
}\left\vert Q^{\prime}(s)\right\vert \beta(s)\left\vert ^{k-5/2}Q^{\prime
}(s^{\prime})\right\vert \left\vert \beta(s^{\prime})\right\vert
^{k-5/2}\left(  \int_{0}^{\widetilde{F}_{-}}\frac{d\rho}{\left(  \beta
(s)-\rho\right)  \left(  \beta(s^{\prime})-\rho\right)  }\right)
dsds^{\prime}%
\]
where $\widetilde{F}_{-}=\min(F_{-}(s),F_{-}(s^{\prime}))$. Denoting the last
integral by $\mathcal{B}_{-}\left(  s,s^{\prime}\right)  $ we obtain in the
case $s\leq s^{\prime}$ that%
\begin{align*}
\mathcal{B}_{-}\left(  s,s^{\prime}\right)   &  =\frac{1}{\beta(s^{\prime
})-\beta(s)}\left.  \left(  \ln\frac{\beta(s^{\prime})-\rho}{\beta(s)-\rho
}\right)  \right\vert _{0}^{F_{-}(s)}\\
&  =-\frac{1}{\beta(s^{\prime})-\beta(s)}\left(  \ln\frac{\beta(s^{\prime}%
)}{\beta(s)}-\ln\frac{\beta(s^{\prime})-\left(  \beta(s)-d\beta^{-\frac{1}{2}%
}(s)\right)  }{d\beta^{-\frac{1}{2}}(s)}\right)  \\
&  =\frac{1}{\beta(s^{\prime})-\beta(s)}\ln\left(  1+\frac{\left(
\beta(s^{\prime})-\beta(s)\right)  \beta^{\frac{1}{2}}(s)}{d}\right)
-\frac{1}{\beta(s^{\prime})-\beta(s)}\ln\frac{\beta(s^{\prime})}{\beta(s)}.
\end{align*}
Similarly to (\ref{B+ ineq}) we obtain%
\[
\mathcal{B}_{-}\left(  s,s^{\prime}\right)  \leq\operatorname*{Const}\left\{
\begin{array}
[c]{c}%
\beta^{\frac{1}{2}}(s),\quad s\leq s^{\prime},\\
\beta^{\frac{1}{2}}(s^{\prime}),\quad s\geq s^{\prime}.
\end{array}
\right.
\]
Thus,%
\[
\left\Vert L_{k}^{\left(  4\right)  }\left(  \rho\right)  \right\Vert
_{L_{2}\left(  \mathbb{R}\right)  }^{2}\leq\operatorname*{Const}\left(
\int_{0}^{\infty}\left\vert Q^{\prime}(s)\right\vert \left(  1+s\right)
^{\frac{k}{2}-1}ds\right)  ^{2}%
\]
and hence, as in the previous cases, under the condition $k\leq2\alpha$ the
inclusion $L_{k}^{\left(  4\right)  }\left(  \rho\right)  \in L_{2}\left(
\mathbb{R}\right)  $ is valid. Thus, if $k\leq2\alpha$ the first two terms in
(\ref{int2}) belong to $L_{\infty}\left(  \mathbb{R}\right)  \cap L_{2}\left(
\mathbb{R}\right)  $. The same condition guarantees the belonging of the other
two terms in (\ref{int2}) to the same class, and hence we proved the belonging
of the integral (\ref{Lk2}) to $L_{\infty}\left(  \mathbb{R}\right)  \cap
L_{2}\left(  \mathbb{R}\right)  $. The integral (\ref{Lk1}) can be studied
analogously, by replacing $k$ with $k+1$ and $\alpha-1$ with $\alpha$. That is
the condition $k\leq2\alpha$ becomes $k\leq2\alpha+1$, and Lemma
\ref{Lemma derivatives in LL} is proved.
\end{proof}

\begin{theorem}
Let the function $\varphi\left(  \rho\right)  $ have the form
(\ref{symbol phi}) and the conditions (\ref{R factorized})--(\ref{Qprime}) be
fulfilled. Then for the solution $y(\rho,x)$ of (\ref{Equation Hankel}) the
inclusion is valid%
\[
\frac{\partial^{k}y\left(  \rho,x\right)  }{\partial x^{k}}\in L_{2}%
^{+}\left(  \mathbb{R}\right)  \quad\text{for }k=0,1,\ldots,\left\lfloor
2\alpha\right\rfloor .
\]

\end{theorem}

\begin{proof}
Due to \cite{GR2015SIAM} equation (\ref{Equation Hankel}) is uniquely
solvable, and the solution belongs to $L_{2}^{+}\left(  \mathbb{R}\right)  $
and has the form (\ref{y}). Let us differentiate (\ref{Equation Hankel}) with
respect to $x$. We have%
\begin{equation}
\left(  \left(  I+\mathbf{H}(\varphi)\right)  \frac{\partial y}{\partial
x}\right)  \left(  \rho\right)  =-\mathbf{H}(\frac{\partial\varphi}{\partial
x})\left(  1\right)  \left(  \rho\right)  -\mathbf{H}(\frac{\partial\varphi
}{\partial x})\left(  y\right)  \left(  \rho\right)  .\label{vsp1}%
\end{equation}
Since $\mathbf{H}(\varphi)=\mathbf{H}(\varphi_{0})T\left(  T_{+}\right)  $
(see (\ref{Hphi factorized})) and $\mathbf{H}(\varphi_{0})=\mathbf{H}%
(\varphi_{0}^{-})$ (see (\ref{phi0-})), then%
\[
\mathbf{H}\left(  \frac{\partial\varphi}{\partial x}\right)  =\mathbf{H}%
\left(  \frac{\partial\varphi_{0}^{-}}{\partial x}\right)  T\left(
T_{+}\right)  .
\]
Due to Lemma \ref{Lemma derivatives in LL}, for $\alpha\geq1$ we have
$\frac{\partial\varphi_{0}^{-}}{\partial x}\in L_{2}\left(  \mathbb{R}\right)
$, and the first term on the right-hand side of\ (\ref{vsp1}) is a function of
the class $L_{2}\left(  \mathbb{R}\right)  $. Additionally, again due to Lemma
\ref{Lemma derivatives in LL}, $\frac{\partial\varphi_{0}^{-}}{\partial x}\in
L_{\infty}\left(  \mathbb{R}\right)  $. Hence
\[
\mathbf{H}\left(  \frac{\partial\varphi}{\partial x}\right)  \left(  y\right)
\left(  \rho\right)  =\mathbf{H}\left(  \frac{\partial\varphi_{0}^{-}%
}{\partial x}\right)  \left(  T_{+}y\right)  \in L_{2}\left(  \mathbb{R}%
\right)  .
\]
Thus, the right-hand side of (\ref{vsp1}) belongs to $L_{2}\left(
\mathbb{R}\right)  $. Hence equation (\ref{vsp1}) is uniquely solvable in
$L_{2}\left(  \mathbb{R}\right)  $, and
\begin{equation}
\frac{\partial y\left(  \rho,x\right)  }{\partial x}\in L_{2}\left(
\mathbb{R}\right)  .\label{dy in L2}%
\end{equation}
Differentiating (\ref{vsp1}) with respect to $x$ we obtain
\[
\left(  \left(  I+\mathbf{H}(\varphi)\right)  \frac{\partial^{2}y}{\partial
x^{2}}\right)  \left(  \rho\right)  =-\mathbf{H}\left(  \frac{\partial
^{2}\varphi}{\partial x^{2}}\right)  \left(  1\right)  \left(  \rho\right)
-\mathbf{H}\left(  \frac{\partial^{2}\varphi}{\partial x^{2}}\right)  \left(
y\right)  \left(  \rho\right)  -2\mathbf{H}\left(  \frac{\partial\varphi
}{\partial x}\right)  \left(  \frac{\partial y}{\partial x}\right)  \left(
\rho\right)  .
\]
Due to Lemma \ref{Lemma derivatives in LL}, $\frac{\partial^{2}\varphi_{0}%
^{-}}{\partial x^{2}}\in L_{\infty}\left(  \mathbb{R}\right)  \cap
L_{2}\left(  \mathbb{R}\right)  $ for $\alpha\geq1$, and since%
\[
\mathbf{H}\left(  \frac{\partial^{2}\varphi}{\partial x^{2}}\right)
=\mathbf{H}\left(  \frac{\partial^{2}\varphi_{0}^{-}}{\partial x^{2}}\right)
T\left(  T_{+}\right)  ,
\]
the first two terms on the right-hand side belong to $L_{2}\left(
\mathbb{R}\right)  $. Due to (\ref{dy in L2}) and the inclusion $\frac
{\partial\varphi_{0}^{-}}{\partial x}\in L_{\infty}\left(  \mathbb{R}\right)
$ the third term belongs to the same space as well. Thus, the right-hand side
here belongs to $L_{2}\left(  \mathbb{R}\right)  $, and hence $\frac
{\partial^{2}y\left(  \rho,x\right)  }{\partial x^{2}}\in L_{2}\left(
\mathbb{R}\right)  $. The case $\ k>2$ is considered analogously.
\end{proof}

\bigskip

\section{Numerical realization of ISTM}

\subsection{Direct scattering\label{SubsectDirectScattering}}

Given a potential $q(x)$, the sets of the coefficients $\left\{
a_{n}\right\}  $, $\left\{  b_{n}\right\}  $, $\left\{  c_{n}\right\}  $ and
$\left\{  d_{n}\right\}  $ are computed following the recurrent integration
procedure from Appendix. The first required for this procedure functions
$e(\frac{i}{2},x)$ and $g(\frac{i}{2},x)$ and their first derivatives can be
computed using any available method for numerical solution of an approximate
Cauchy problem for the equation
\begin{equation}
-y^{\prime\prime}+q(x)y+\frac{1}{4}y=0. \label{Schr1/4}%
\end{equation}
In particular, we used the SPPS (spectral parameter power series) method from
\cite{KrPorter2010} (see also \cite{APFT} and \cite{KrBook2020}). To find an
approximation $\widetilde{e}\left(  \frac{i}{2},x\right)  $ of $e\left(
\frac{i}{2},x\right)  $ one can compute a solution of the Cauchy problem for
(\ref{Schr1/4}) on a sufficiently large interval $\left(  0,b\right)  $ with
the initial conditions
\[
\widetilde{e}\left(  \frac{i}{2},b\right)  =e^{-\frac{b}{2}}\quad
\text{and}\quad\widetilde{e}\,^{\prime}\left(  \frac{i}{2},b\right)
=-\frac{e^{-\frac{b}{2}}}{2}%
\]
which follow from the asymptotics of the Jost solution.

Next, $\eta(x)$ is computed by (\ref{Abel formula}), see Appendix. Analogously
$g(\frac{i}{2},x)$ and $\xi(x)$ (see Appendix) are computed, and a number of
the coefficients $\left\{  a_{n}\right\}  $, $\left\{  b_{n}\right\}  $,
$\left\{  c_{n}\right\}  $ and $\left\{  d_{n}\right\}  $ is computed
following the recurrent integration procedure from Appendix. For the recurrent
integration we use the Newton-Cottes 6-point integration rule. Having computed
the sets of the coefficients, the scattering data are computed with the aid of
the corresponding formulas from Section
\ref{Sect Representations scattering data}.

\subsection{Inverse scattering\label{SubsectInverseScattering}}

The following numerical method for solving the inverse scattering problem can
be proposed.

\begin{enumerate}
\item Given a set of scattering data $J^{+}$ or $J^{-}$. Choose a number of
equations $N_{s}$, so that the truncated system
\begin{equation}
a_{m}(x)+\sum_{n=0}^{N_{s}}a_{n}(x)A_{mn}(x)=r_{m}(x),\quad m=0,\ldots,N_{s}
\label{truncated system A}%
\end{equation}
or
\begin{equation}
b_{m}(x)+\sum_{n=0}^{N_{s}}b_{n}(x)B_{mn}(x)=s_{m}(x),\quad m=0,\ldots,N_{s}
\label{truncated system B}%
\end{equation}
is to be solved.

\item Compute $r_{m}(x)$ and $A_{mn}(x)$ or $s_{m}(x)$ and $B_{mn}(x)$
according to the formulas from the preceding section.

\item Solve the system (\ref{truncated system A}) or (\ref{truncated system B}%
) to find $a_{0}(x)$ or $b_{0}(x)$, respectively.

\item Compute $q$ with the aid of (\ref{q=a}) or (\ref{q=b})
\end{enumerate}

Let us discuss some relevant aspects of this numerical approach.

The first question is regarding an appropriate way for computing the integrals
in (\ref{alpha mn})-(\ref{s m}). Notice that they are Fourier transforms of
the reflection coefficients multiplied by fractions of a special form. Here it
would be interesting to apply some of the available techniques for numerical
computation of the Fourier transform. However, the special form of the
fractional factors suggests another possibility for computing the integrals,
which proved to provide good results.

The integral we are interested in computing has the form%
\[
I(x)=\int_{-\infty}^{\infty}s\left(  \rho\right)  e^{2i\rho x}\frac{\left(
\frac{1}{2}+i\rho\right)  ^{m}}{\left(  \frac{1}{2}-i\rho\right)  ^{n}}d\rho
\]
where $n-m\geq1$. Let $z:=\frac{\frac{1}{2}+i\rho}{\frac{1}{2}-i\rho}$. Then,
when $\rho$ runs along the real axis $(-\infty,\infty)$, the variable $z$ runs
along the unitary circle, so that $z=e^{i\theta}$ with $\theta\in(-\pi,\pi)$.
Thus, the change of the integration variable has the form $e^{i\theta}%
=\frac{\frac{1}{2}+i\rho}{\frac{1}{2}-i\rho}$. Then $e^{i\theta}d\theta
=\frac{d\rho}{\left(  \frac{1}{2}-i\rho\right)  ^{2}}$, $2i\rho=\frac
{e^{i\theta}-1}{e^{i\theta}+1}$ and $\frac{1}{\frac{1}{2}-i\rho}=e^{i\theta
}+1$. Hence the integral $I(x)$ can be written as
\begin{equation}
I(x)=\int_{-\pi}^{\pi}s\left(  \frac{i(1-e^{i\theta})}{2(1+e^{i\theta}%
)}\right)  \exp\left(  \frac{e^{i\theta}-1}{e^{i\theta}+1}x\right)
e^{i\theta(m+1)}\left(  e^{i\theta}+1\right)  ^{n-m-2}d\theta. \label{Ix}%
\end{equation}
All the integrals involved were computed using formula (\ref{Ix}) with the aid
of the Matlab routine `trapz', evaluating the integrand in $N_{i}=10^{4}$
points, uniformly distributed on the interval $(-\pi,\pi)$.

Finally, on the last step, for recovering $q$ with the aid of (\ref{q=a}) or
(\ref{q=b}), the computed coefficient $a_{0}$ (or $b_{0}$) needs to be
differentiated twice. This was performed by representing the computed
coefficient in the form of a spline with the aid of the Matlab routine `spapi'
with a posterior differentiation with the Matlab command `fnder'.

\subsection{Numerical examples}

\begin{example}
Consider the Cauchy problem for the KdV equation (\ref{KdV}) with the initial
data
\begin{equation}
q(x)=xe^{-x^{2}}. \label{q ex1}%
\end{equation}
On Figure \ref{FigLamb} we reproduce the results from one of the first books
on the subject \cite{Lamb}, where on p. 117 the reader can find a similar
figure with one important difference. The graphs from \cite{Lamb} contain some
additional oscillations at the ends of the interval depicted which as
explained by the author is due to imposition of (artificial) periodic boundary
conditions. Application of our method does not require imposing any artificial
boundary conditions.%

\begin{figure}
[ptb]
\centering
\includegraphics[
height=3in,
width=5in
]
{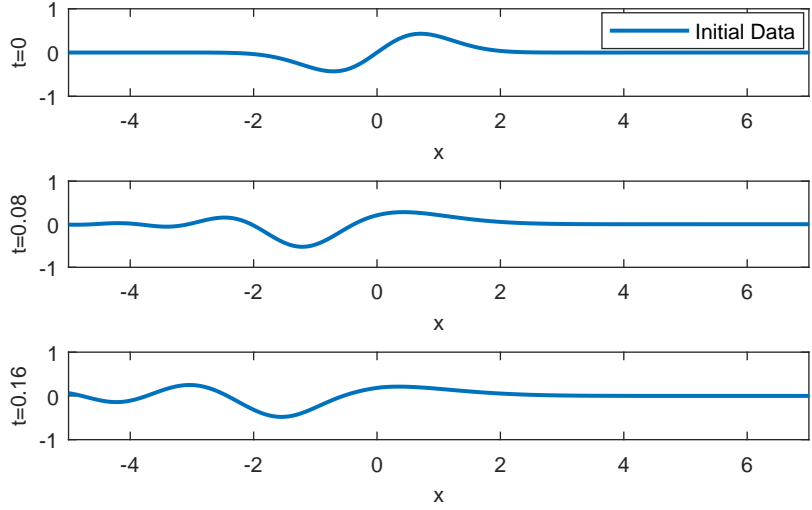}%
\caption{Solution of the Cauchy problem for (\ref{KdV}) with the initial data
(\ref{q ex1}).}%
\label{FigLamb}%
\end{figure}

The potential (\ref{q ex1}) has one eigenvalue which was computed numerically,
$\lambda_{1}\approx-0.0138384593995$ with the corresponding norming constants
$\alpha_{1}^{-}\approx0.2055954681199$ and $\alpha_{1}^{+}\approx
0.0416040800785$. The reflection coefficients $s^{-}(\rho)$ and $s^{+}(\rho)$
where computed as explained in subsection \ref{SubsectDirectScattering} on the
interval $\left(  -500,500\right)  $. This computation of the scattering data
took less than a second in Matlab2017 on a Laptop computer equipped with an
Intel Core i7 processor.

The inverse scattering problems were solved as explained in subsection
\ref{SubsectInverseScattering} with five equations in the truncated systems
(\ref{truncated system A}) used for $x<0$ and (\ref{truncated system B}) for
$x>0$. The elapsed time to obtain Figure \ref{FigLamb} was 9 sec.

Notice that after solving the direct scattering problem and the inverse
scattering problem for $t=0$ the potential (\ref{q ex1}) on the interval
$\left(  -5,7\right)  $ was recovered with the absolute error of $1.5\cdot%
10^{-3}$.
\end{example}

\begin{example}
Consider the Cauchy problem for the KdV equation (\ref{KdV}) with the initial
data in the form of the reflectionless potential
\begin{equation}
q(x)=-\frac{c}{2}\operatorname*{sech}{}^{2}\left(  \frac{\sqrt{c}x}{2}\right)
. \label{q ex3}%
\end{equation}
The corresponding solution is the solitary wave
\[
u(x,t)=-\frac{c}{2}\operatorname*{sech}{}^{2}\left(  \frac{\sqrt{c}\left(
x-ct\right)  }{2}\right)  .
\]

On Figure \ref{FigSoliton} the computed solution is presented for $c=\pi$ and
$t=0,0.5,1$. Numerical solution of the direct scattering problem gave us the
reflection coefficient equal to zero with the accuracy $2\cdot10^{-4}$, and
the only eigenvalue $\lambda_{1}\approx-0.78539816329$, while its exact value
is $-c/4=-0.78539816339\ldots$ (the difference in the 10-th digit). The
corresponding norming constants are $\alpha_{1}^{\pm}=\sqrt{\pi}$ which were
computed with the accuracy $2.9\cdot10^{-5}$. The solutions of the inverse
scattering problems were obtained with five equations in the truncated systems
(\ref{truncated system A}) used for $x<0$ and (\ref{truncated system B}) for
$x>0$.
\begin{figure}
[ptb]
\centering
\includegraphics[
height=3in,
width=5in
]%
{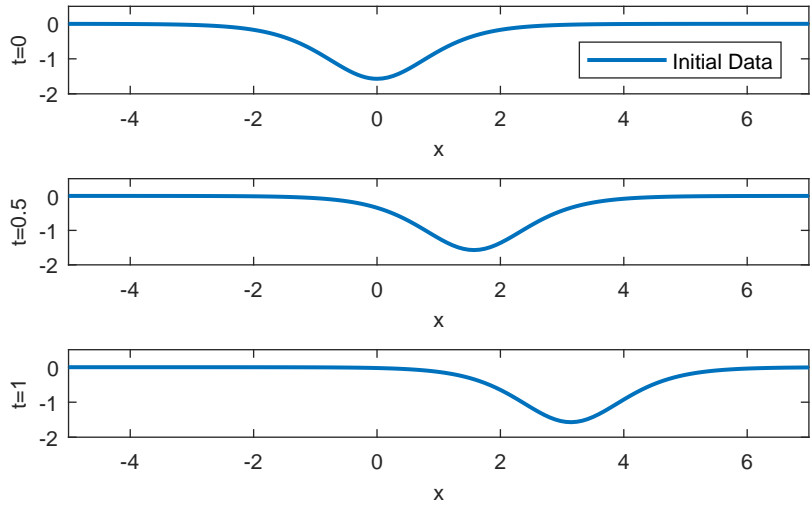}%
\caption{Solution of the Cauchy problem for (\ref{KdV}) with the initial data
(\ref{q ex3}).}%
\label{FigSoliton}%
\end{figure}

After solving the direct scattering problem and the inverse scattering problem
for $t=0$ the potential on the interval $\left(  -5,7\right)  $ was recovered
with the absolute error of $8\cdot%
10^{-4}$, and for subsequent times the error did not grow. For example, for
$t=1$ the absolute error was $2.4\cdot%
10^{-4}$.
\end{example}

\begin{example}
Consider the Cauchy problem for the KdV equation (\ref{KdV}) with the initial
data
\begin{equation}
q(x)=\left\{
\begin{array}
[c]{cc}%
e^{x}\cos4x, & x<0,\\
e^{-x}J_{0}(2x), & x>0,
\end{array}
\right.  \label{q ex2}%
\end{equation}
where $J_{0}(z)$ stands for the Bessel function of the first kind of order
zero. The graph of this potential is presented on Figure \ref{FigPridumannij}.
Its first derivative is discontinuous.
\begin{figure}
[ptb]
\centering
\includegraphics[
height=2.4in,
width=4in
]%
{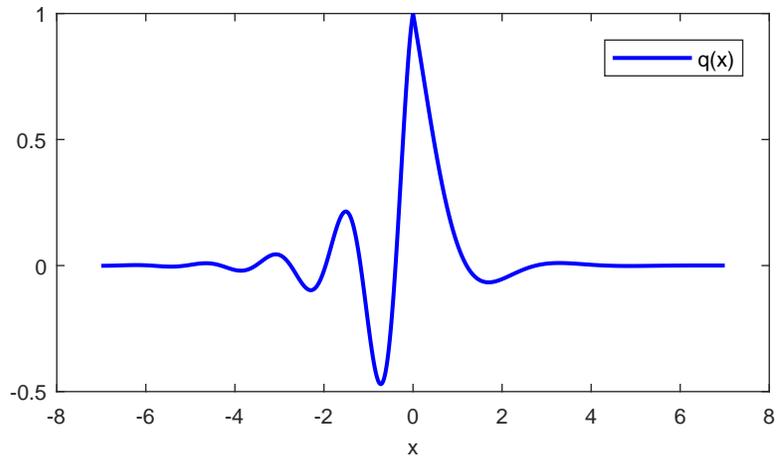}%
\caption{Graph of the potential (\ref{q ex2}).}%
\label{FigPridumannij}%
\end{figure}

On Figure the corresponding solution of the Cauchy problem for (\ref{KdV}) is
presented. It was obtained with nine equations in the truncated systems (with
five equations it was only slightly less accurate). After solving the direct
scattering problem and the inverse scattering problem for $t=0$ the potential
(\ref{q ex2}) on the interval $\left(  -7,7\right)  $ was recovered with the
absolute error of $6\cdot10^{-3}$.%

\begin{figure}
[ptb]
\begin{center}
\includegraphics[
height=3in,
width=5in
]%
{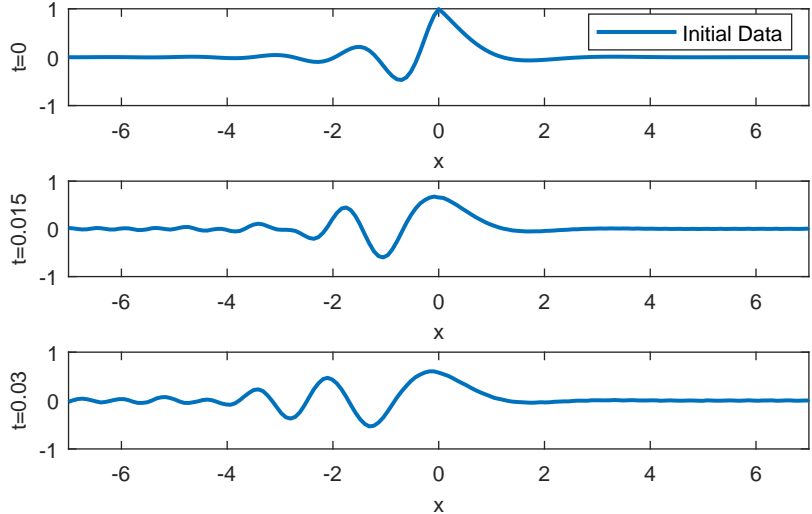}%
\caption{Solution of the Cauchy problem for (\ref{KdV}) with the initial data
(\ref{q ex2}).}%
\label{FigPridumannijEvolution}%
\end{center}
\end{figure}

\end{example}

\section*{Conclusions}

A method for practical realization of the inverse scattering transform method
for the Korteweg--de Vries equation is proposed. It is accurate, efficient
and relatively simple for implementation. The applicability of the finite
section method to the system of linear algebraic equations, which arises when
solving the inverse scattering problem, is proved. Numerical examples are
provided. Similar approach can be derived for other integrable non-linear
equations, such as the nonlinear Schr\"{o}dinger equation.

\section*{Acknowledgements}
Research of Vladislav Kravchenko and Sergii Torba was supported by CONACYT,
Mexico via the project 284470. Research of Sergei Grudsky was supported by
CONACYT, Mexico via the project \textquotedblleft Ciencia de
Frontera\textquotedblright\ FORDECYT-PRONACES/61517/2020. Research of Sergei
Grudsky and Vladislav Kravchenko was performed at the Regional mathematical
center of the Southern Federal University with the support of the Ministry of
Science and Higher Education of Russia, agreement  075--02--2022--893.

\appendix\section{Recurrent integration procedure for the
coefficients of the representations}

Here we recall the result from \cite{DelgadoKhmelnytskayaKrHalfline}
concerning the coefficients $\left\{  a_{n}\right\}  $ and $\left\{
d_{n}\right\}  $ and obtain by analogy similar relations for the coefficients
$\left\{  b_{n}\right\}  $ and $\left\{  c_{n}\right\}  $. Denote%

\begin{equation}
\eta(x):=e(\frac{i}{2},x)\int_{0}^{x}\frac{dt}{e^{2}(\frac{i}{2},t)}
\label{Abel formula}%
\end{equation}
and
\begin{equation}
\xi(x):=g(\frac{i}{2},x)\int_{x}^{0}\frac{dt}{g^{2}(\frac{i}{2},t)}.
\label{Abel formula 2}%
\end{equation}
These solutions of (\ref{Schr}) with $\lambda=-1/4$ satisfy the asymptotic
relations $\eta(x)=e^{\frac{x}{2}}\left(  1+o(1)\right)  $, $x\rightarrow
\infty$ and $\xi(x)=e^{-\frac{x}{2}}\left(  1+o(1)\right)  $, $x\rightarrow
-\infty$, respectively. Computation of all coefficients can be performed
acording to the following steps.

\begin{enumerate}
\item The coefficients $a_{0}$, $b_{0}$, $c_{0}$ and $d_{0}$ have the form%
\begin{align}
a_{0}(x)&=e(\frac{i}{2},x)e^{\frac{x}{2}}-1\quad\text{and}\quad b_{0}%
(x)=g(\frac{i}{2},x)e^{-\frac{x}{2}}-1,\nonumber\\
d_{0}(x)&=a_{0}^{\prime}(x)-\frac{a_{0}(x)}{2}+\frac{1}{2}\int_{x}^{\infty
}q(t)dt, \label{d0}\\
c_{0}(x)&=b_{0}^{\prime}(x)+\frac{b_{0}(x)}{2}-\frac{1}{2}\int_{-\infty}%
^{x}q(t)dt\nonumber
\end{align}

\item The subsequent coefficients $a_{n}$ and $b_{n}$ together with their
derivatives can be computed as follows
\begin{align}
a_{n}(x)&=a_{0}(x)-2e^{\frac{x}{2}}\left(  \eta(x)J_{1,n}(x)-e(\frac{i}%
{2},x)J_{2,n}(x)\right)  ,\nonumber\\
b_{n}(x)&=b_{0}(x)+2e^{-\frac{x}{2}}\left(  \xi(x)I_{1,n}(x)-g(\frac{i}%
{2},x)I_{2,n}(x)\right)  ,\nonumber\\
a_{n}^{\prime}(x)  &  =a_{0}^{\prime}(x)-2\left(  e^{\frac{x}{2}}%
\eta(x)\right)  ^{\prime}J_{1,n}(x)-2e^{\frac{x}{2}}\eta(x)J_{1,n}^{\prime
}(x)\nonumber\\
&  +2\left(  e^{\frac{x}{2}}e(\frac{i}{2},x)\right)  ^{\prime}J_{2,n}%
(x)+2e^{\frac{x}{2}}e(\frac{i}{2},x)J_{2,n}^{\prime}(x), \label{an prime}\\
b_{n}^{\prime}(x)  &  =b_{0}^{\prime}(x)+2\left(  e^{-\frac{x}{2}}%
\xi(x)\right)  ^{\prime}I_{1,n}(x)+2e^{-\frac{x}{2}}\xi(x)I_{1,n}^{\prime
}(x)\nonumber\\
&  -2\left(  e^{-\frac{x}{2}}g(\frac{i}{2},x)\right)  ^{\prime}I_{2,n}%
(x)-2e^{-\frac{x}{2}}g(\frac{i}{2},x)I_{2,n}^{\prime}(x),\nonumber
\end{align}
where
\begin{equation}
J_{1,n}(x)=J_{1,n-1}(x)-e^{-\frac{x}{2}}e(\frac{i}{2},x)a_{n-1}(x)-\int
_{x}^{\infty}\left(  e(\frac{i}{2},t)e^{-\frac{t}{2}}\right)  ^{\prime}%
a_{n-1}(t)dt, \label{J1 recurrent}%
\end{equation}
\begin{equation}
J_{2,n}(x)=J_{2,n-1}(x)-e^{-\frac{x}{2}}\eta(x)a_{n-1}(x)-\int_{x}^{\infty
}\left(  \eta(t)e^{-\frac{t}{2}}\right)  ^{\prime}a_{n-1}(t)dt
\label{J2 recurrent}%
\end{equation}
\begin{equation}
I_{1,n}(x)=I_{1,n-1}(x)+e^{\frac{x}{2}}g(\frac{i}{2},x)b_{n-1}(x)-\int
_{-\infty}^{x}\left(  g(\frac{i}{2},t)e^{\frac{t}{2}}\right)  ^{\prime}%
b_{n-1}(t)dt, \label{I1}%
\end{equation}
and
\begin{equation}
I_{2,n}(x)=I_{2,n-1}(x)+e^{\frac{x}{2}}\xi(x)b_{n-1}(x)-\int_{-\infty}%
^{x}\left(  \xi(t)e^{\frac{t}{2}}\right)  ^{\prime}b_{n-1}(t)dt. \label{I2}%
\end{equation}
($J_{1,0}(x)=J_{2,0}(x)=I_{1,0}(x)=I_{2,0}(x)\equiv0$), and for $J_{1,n}%
^{\prime}(x)$, $J_{2,n}^{\prime}(x)$, $I_{1,n}(x)$ and $I_{2,n}(x)$ the
following relations are obtained from (\ref{J1 recurrent})-(\ref{I2}),%
\begin{equation}
J_{1,n}^{\prime}(x)=J_{1,n-1}^{\prime}(x)-e^{-\frac{x}{2}}e(\frac{i}%
{2},x)a_{n-1}^{\prime}(x), \label{J1 prime}%
\end{equation}
\begin{equation}
J_{2,n}^{\prime}(x)=J_{2,n-1}^{\prime}(x)-e^{-\frac{x}{2}}\eta(x)a_{n-1}%
^{\prime}(x), \label{J2 prime}%
\end{equation}%
\[
I_{1,n}^{\prime}(x)=I_{1,n-1}^{\prime}(x)+e^{\frac{x}{2}}g(\frac{i}%
{2},x)b_{n-1}^{\prime}(x)
\]
and%
\[
I_{2,n}^{\prime}(x)=I_{2,n-1}^{\prime}(x)+e^{\frac{x}{2}}\xi(x)b_{n-1}%
^{\prime}(x).
\]

\item The coefficients $d_{n}$ and $c_{n}$ can be computed with the aid of the
relations
\begin{equation}
d_{n+1}(x)=d_{n}(x)+a_{n+1}^{\prime}(x)-a_{n}^{\prime}(x)-\frac{1}{2}\left(
a_{n+1}(x)+a_{n}(x)\right)  ,\quad n=0,1,\ldots\label{dn recurrent}%
\end{equation}
and
\[
c_{n+1}(x)=c_{n}(x)+b_{n+1}^{\prime}(x)-b_{n}^{\prime}(x)+\frac{1}{2}\left(
b_{n+1}(x)+b_{n}(x)\right)  ,\quad n=0,1,\ldots.
\]

\end{enumerate}

\end{document}